\documentclass[11pt,a4paper,twosides]{amsart}
\usepackage{amssymb,amsmath,amsthm,mathrsfs,graphicx,xcolor}
\usepackage{amsrefs}
\usepackage{leftidx}
\usepackage{hyperref}

\theoremstyle{plain}
\newtheorem{theorem}{Theorem}[section]

\newtheorem{corollary}[theorem]{Corollary}
\newtheorem{proposition}[theorem]{Proposition}
\newtheorem{assumption}{Assumption} 
\theoremstyle{remark}
\newtheorem{remark}[theorem]{Remark}

\allowdisplaybreaks[4]

\numberwithin{equation}{section}

\newcommand{\C}{\mathbb{C}}
\newcommand{\R}{\mathbb{R}}
\newcommand{\Z}{\mathbb{Z}}
\newcommand{\N}{\mathbb{N}}

\renewcommand{\Im}{\operatorname{Im}}
\renewcommand{\Re}{\operatorname{Re}}

\def\({\left(}
\def\){\right)}
\def\<{\left\langle}
\def\>{\right\rangle}
\def\le{\leqslant}
\def\ge{\geqslant}

\newcommand{\eps}{\varepsilon}

\newcommand{\pt}{\partial}

\DeclareMathOperator{\sign}{sign}

\DeclareMathOperator{\rank}{rank}
\DeclareMathOperator{\tr}{tr}
\DeclareMathOperator{\Span}{span}

\newcommand{\todayd}{\the\year/\the\month/\the\day}

\theoremstyle{definition}

\newcommand{\ol}{\overline}
\renewcommand{\phi}{\varphi}
\renewcommand{\epsilon}{\varepsilon}
\renewcommand{\tilde}{\widetilde}

\begin{document}
\title[]
{Non-Polynomial Conserved Quantities for ODE Systems and its Application to the Long-Time Behavior of Solutions to Cubic NLS Systems}

\author{Satoshi Masaki}
\address{Department of Mathematics, 
Faculty of Science, Hokkaido University,
Kita 10, Nishi 8, Kita-Ku, Sapporo, Hokkaido, 060-0810, Japan}
\email{masaki@math.sci.hokudai.ac.jp}

\author{Jun-ichi Segata}
\address{Faculty of Mathematics, Kyushu University, 
Fukuoka, 819-0395, Japan}
\email{segata@math.kyushu-u.ac.jp}

\author{Kota Uriya}
\address{Department of Applied Mathematics, Faculty of Science, 
Okayama University of Science, Okayama, 700-0005, Japan}
\email{uriya@xmath.ous.ac.jp}


\maketitle

\begin{abstract}
%
%
In this paper, we investigate the asymptotic behavior of small solutions to the initial value problem for a system of cubic nonlinear Schr\"odinger equations (NLS) in one spatial dimension. We identify a new class of NLS systems for which the global boundedness and asymptotics of small solutions can be established, even in the absence of any effective conserved quantity. The key to this analysis lies in utilizing conserved quantities for the reduced ordinary differential equation (ODE) systems derived from the original NLS systems. In a previous study, the first author investigated conserved quantities expressed as quartic polynomials. In contrast, the conserved quantities considered in the present paper are of a different type and are not necessarily polynomial.
\end{abstract}

\section{Introduction}
In this paper, we consider the large time behavior of solutions to
the initial value problem of the system of cubic nonlinear Schr\"odinger equations in one space dimension:
\begin{equation}
  \label{eq:2NLS}
  \left\{
  \begin{aligned}
     & i\partial_t u_1 + \frac12 \partial_x^2 u_1 =
    F_1(u_1,u_2), \quad t \in \R,\ x \in \R,                                \\
     & i\partial_t u_2 + \frac12 \partial_x^2 u_2 =
    F_2(u_1,u_2),
    \quad t \in \R,\ x \in \R,                                              \\
     & u_1(0,x) = u_{1,0}(x), \quad u_2(0,x) = u_{2,0}(x), \qquad x \in \R,
  \end{aligned}
  \right.
\end{equation}
where $u_j:\R\times\R\to\C$ ($j=1,2$) are unknown functions
and $u_{j,0}:\R\to\C$ ($j=1,2$) are given functions. The nonlinear terms
$F_1$ and $F_2$ are given by
\begin{equation*}
  \begin{aligned}
    F_1(u_1,u_2)
     & = c_1|u_1|^2u_1 + c_2|u_1|^2u_2 + c_3u_1^2\ol{u_2}
    + c_4 u_1|u_2|^2 + c_5 \ol{u_1}u_2^2 + c_6|u_2|^2u_2,          \\
    F_2(u_1,u_2)
     & = c_7|u_1|^2u_1 + c_8|u_1|^2u_2 + c_9u_1^2\ol{u_2}
    + c_{10} u_1|u_2|^2 + c_{11} \ol{u_1}u_2^2 + c_{12}|u_2|^2u_2, \\
  \end{aligned}
\end{equation*}
and
$ c_j\ (j = 1, \cdots, 12)$ are real constants.
Our aim is to show the global existence of the solution 
to (\ref{eq:2NLS}) for small initial data and 
to derive their asymptotic behavior under suitable conditions on the parameters of 
the nonlinearities. It is well-known that the cubic nonlinearities are critical 
in determining the asymptotic behavior of solutions in one dimension, making this an important and intriguing 
problem. 
We provide two new types of systems of the form \eqref{eq:2NLS}. A typical examples in our scope are
\begin{equation}
  \label{eq:2NLSa}
  \left\{
  \begin{aligned}
     & i\partial_t u_1 + \frac12 \partial_x^2 u_1 =
    (\zeta+1) |u_1|^2 u_2 +\zeta |u_2|^2u_2, \quad t \in \R,\ x \in \R,                                \\
     & i\partial_t u_2 + \frac12 \partial_x^2 u_2 =
    2\zeta |u_1|^2 u_1+ 2\zeta u_1 |u_2|^2 +\ol{u_1}u_2^2 ,
    \quad t \in \R,\ x \in \R                                                \end{aligned}
  \right.
\end{equation}
with $\zeta \in \R\setminus\{0\}$
and
\begin{equation}
  \label{eq:2NLSb}
  \left\{
  \begin{aligned}
     & i\partial_t u_1 + \frac12 \partial_x^2 u_1 =
    6|u_1|^2u_1 +3|u_1|^2u_2 +u_1^2\ol{u_2}, \quad t \in \R,\ x \in \R,                                \\
     & i\partial_t u_2 + \frac12 \partial_x^2 u_2 =
    -u_1|u_2|^2 - 4 |u_2|^2 u_2,
    \quad t \in \R,\ x \in \R.
    \end{aligned}
  \right.
\end{equation}

Considering the single gauge invariant cubic nonlinear Schr\"odinger equation:
\begin{equation}\label{eq:NLS}
  i\partial_t u + \frac{1}{2}\partial_x^2 u = \lambda|u|^2u, \quad t \in \R,\  x \in \R
\end{equation}
with $\lambda \in \C$, 
two typical asymptotic behavior of the solution are known, depending on whether $\lambda \in \C$ is a 
real number or a pure imaginary number. 
When $\lambda$ is a real number, Ozawa~\cite{Oz} showed the modified scattering of the small solution to \eqref{eq:NLS}, 
namely, for a suitable function $u_+$ depending only on the spatial variable, 
\begin{equation}\label{eq:msca}
  u(t,x) = (it)^{-\frac12}e^{\frac{i|x|^2}{2t}}\widehat{u_+}\left(\frac{x}{t}\right)
  \exp(-i\lambda|\widehat{u_+}(\tfrac{x}{t})|\log t) + o(t^{-\frac12})
\end{equation}
in $L^\infty$ as $t \to \infty$ 
(see also \cites{GiOz,HaNa,HaNa2,IfTa,KaPu,LS}). 
When $\lambda$ is a pure imaginary number, 
Shimomura~\cite{Shi} showed that the equation has the dissipative property in one time direction 
(See also \cite{KiSh,HaNaSu}). 
Kita~\cite{Ki} showed a counter part of the above result, namely small data blow-up occurs in converse time direction. 


 
In studying the general system of cubic nonlinear Sch\"odinger equations \eqref{eq:2NLS}, we encounter the 
various situations depending on the coefficients of the nonlinearity 
\cite{KaSa,Kim,LiNiSaSu1,LiNiSaSu2,LiNiSaSu5,LiSu,LiSu2,MuPu,Ur}. 
In the our previous studies~\cite{MaSeUr1,MaSeUr2}, 
we provided the classification of the system \eqref{eq:2NLS}. 
First, we review it in detail. 
In \cite{MaSeUr1}, we showed that the system \eqref{eq:2NLS} can be identified with a pair of a matrix $\mathscr{A} \in M_3(\R) \simeq \R^9$
and a vector $\mathscr{V} \in \R^3$ as follows:
Given $(c_1, \cdots, c_{12})\in \R^{12}$,
we define
\begin{equation}
  \mathscr{A} :=
  \begin{bmatrix}
    c_2 - c_3 & -c_1 + c_8 - c_9    & - c_7                  \\
    c_5             & -c_3 +c_{11}              & -c_9                   \\
    c_6             & -c_4 + c_5 + c_{12} & -c_{10} + c_{11}
  \end{bmatrix}
\end{equation}
and
\begin{equation}
  \mathscr{V}:=
  \begin{bmatrix}
    c_8-2c_9                                      \\
    \frac12(-c_2+2c_3-c_{10}+2c_{11}) \\
    c_4-2c_5
  \end{bmatrix}.
\end{equation}
Conversely, 
for a given pair $(\mathscr{A}=(a_{ij})_{1\le i,j \le 3}, (q_k)_{1\le k \le 3} )\in M_3(\R) \times \R^3$, the system of the form \eqref{eq:2NLS} 
is given with the following nonlinearity (see \cite{MaSeUr1});
\begin{equation}\label{E:CVNLS}
\left\{
\begin{aligned}
	F_1(u_1,u_2)={}&
-(a_{12}+a_{23}) |u_1|^2 u_1 +a_{11}(2|u_1|^2 u_2 + u_1^2\ol{u_2})+a_{21} (2 u_1 |u_2|^2 + \ol{u_1}u_2^2) \\&+ a_{31} |u_2|^2u_2
	 - (\tr \mathscr{A}) \Re (\overline{u_1} u_2) u_1 +\mathcal{V}(u_1,u_2) u_1, \\
	F_2(u_1,u_2)={}&
-a_{13} |u_1|^2 u_1-a_{23}(2|u_1|^2 u_2 + u_1^2\ol{u_2})-a_{33} (2 u_1 |u_2|^2 + \ol{u_1}u_2^2) \\&+ (a_{21}+a_{32}) |u_2|^2u_2
	 + (\tr \mathscr{A}) \Re (\overline{u_1} u_2) u_2 + \mathcal{V}(u_1,u_2)u_2,
\end{aligned}
\right.
\end{equation}
where $\mathcal{V}(u_1,u_2)=q_1 |u_1|^2 + 2q_2 \Re (\overline{u_1}u_2)  + q_3 |u_2|^2$ is a real-valued quadratic potential.
The matrix $\mathscr{A}$ naturally appears in the analysis of the following system of ordinary differential equations, which is obtained 
 by removing the term $\frac12\partial_x^2u_j$ from the system \eqref{eq:2NLS}:
\begin{equation}\label{eq:ODE}
  \left\{
  \begin{aligned}
    i\partial_t \phi_1 = F_1(\phi_1, \phi_2), \\
    i\partial_t \phi_2 = F_2(\phi_1, \phi_2).
  \end{aligned}
  \right.
\end{equation}
The system \eqref{eq:ODE} is called the \textit{reduced ODE system}. We introduce the four quadratic quantities 
\begin{equation}\label{eq:rhoq}
  \rho_1 = |\phi_1|^2,\quad \rho_2=|\phi_2|^2,\quad  \mathcal{R} = 2\Re(\overline{\phi_1}\phi_2),\quad  \mathcal{I}=2\Im(\overline{\phi_1}\phi_2)
\end{equation}
for a solution $(\varphi_1,\varphi_2)$ to \eqref{eq:ODE}. 
Further, for a vector $\vec{a}={}^t\!(a_1,a_2,a_3)\in\R^3$ and a solution $(\varphi_1,\varphi_2)$ to \eqref{eq:ODE}, we define a quadratic form
\begin{equation}\label{eq:Q}
	\mathcal{Q}(\vec a) := \begin{bmatrix}
      \rho_1 & \mathcal{R} & \rho_2
    \end{bmatrix} \vec{a}= \begin{bmatrix}
    \overline{\varphi_1} &
    \overline{\varphi_2}
    \end{bmatrix}
    \begin{bmatrix}
      a_1 & a_2 \\ a_2 & a_3
    \end{bmatrix}
   \begin{bmatrix}
   \varphi_1 \\ \varphi_2
	\end{bmatrix}.
\end{equation}
Following proposition provides the time evolution of the 
quadratic form $\mathcal{Q}(\vec{a})$ along the flow given by \eqref{eq:ODE}.
\begin{proposition}[\cite{MaSeUr2}] \label{prop:ode}
  For any $\vec{a}\in \R^3$, it holds that
  \begin{equation}
    \frac{d}{dt} \mathcal{Q}(\vec{a})
    = 2\mathcal{I} \mathcal{Q}(\mathscr{A}\vec{a}).
  \end{equation}
  In particular, if $\vec{a} \in \ker \mathscr{A}$ then $\mathcal{Q}(\vec{a})$ is independent of time.
\end{proposition}
We introduce the following sets:
\begin{equation}
\begin{aligned}
	\mathcal{P}_+ := \{ {}^t\!(a,b,c) \in \R^3 \ |\ ac>b^2\} \subset \R^3,\\
	\mathcal{P}_0 := \{ {}^t\!(a,b,c) \in \R^3 \ |\ ac=b^2\} \subset \R^3,\\
	\mathcal{P}_- := \{ {}^t\!(a,b,c) \in \R^3 \ |\ ac<b^2\} \subset \R^3.
\end{aligned}
\end{equation}
Note that if $\vec{a} \in \mathcal{P}_+$ then the quadratic form $\mathcal{Q}(\vec{a})$
is sign-definite, i.e., $|\mathcal{Q}(\vec{a})| \sim_{\vec{a}} |\varphi_1|^2 + |\varphi_2|^2$.

From Proposition \ref{prop:ode}, we see that $\rank \mathscr{A}$ indicates the complexity of the system since
the number of the conserved quadratic quantities of the reduced ODE system \eqref{eq:ODE}
is $3 - \rank \mathscr{A}$. 
In \cite{MaSeUr2,Ma}, 
the system \eqref{eq:2NLS} satisfying $\rank \mathscr{A}\le 2$ are classified.
Roughly speaking, there are the following seven types;
\begin{enumerate}
\item $\rank \mathscr{A}=0$;
\item $\rank \mathscr{A}=1$ and $\ker \mathscr{A} \cap \mathcal{P}_+\neq \emptyset$;
\item $\rank \mathscr{A}=1$, $\ker \mathscr{A} \cap \mathcal{P}_+= \emptyset$, and $\ker \mathscr{A} \cap \mathcal{P}_0\neq \{0\}$;
\item $\rank \mathscr{A}=1$, $\ker \mathscr{A} \cap \mathcal{P}_+= \emptyset$, and $\ker \mathscr{A} \cap \mathcal{P}_0= \{0\}$;
\item
$\rank \mathscr{A}=2$ and $\ker \mathscr{A} \cap \mathcal{P}_+\neq \emptyset$;
\item
$\rank \mathscr{A}=2$ and $\ker \mathscr{A} \cap \mathcal{P}_0\neq \{0\}$;
\item
$\rank \mathscr{A}=2$ and $\ker \mathscr{A} \cap \mathcal{P}_-\neq \emptyset$.
\end{enumerate}
The case (1) includes the Manakov system \cite{Manakov} 
which is an integrable system arising in the nonlinear optics.

In the cases (1), (2) and (5),
a structural condition called the \emph{weak null gauge condition} is satisfied. 
As for \eqref{eq:2NLS}, this condition is expressed as
\begin{equation}\label{eq:WNGC}
	\ker \mathscr{A} \cap \mathcal{P}_+ \neq \emptyset
\end{equation}
(see \cite{Ma2}).
If this condition is satisfied, 
then one can find a conserved quadratic form which is equivalent to
$|\varphi_1(t)|^2 + |\varphi_2(t)|^2$ in view of Proposition \ref{prop:ode}.
Hence, solutions to the reduced ODE system  \eqref{eq:ODE} are global and bounded. 
Kim \cite{Kim} studied the global existence and time decay of solution to 
\eqref{eq:2NLS} under the weak null gauge condition (see also Li-Sunagawa \cite{LiSu,LiSu2} 
for \eqref{eq:2NLS} with derivative cubic nonlinearity). Furthermore, 
Katayama-Sakoda~\cite{KaSa} 
proved that, under the weak null gauge condition, the 
asymptotic behavior of the solution of the system 
\eqref{eq:2NLS} is given by
\begin{equation}
  u_j(t,x) = (it)^{-\frac12}e^{\frac{i|x|^2}{2t}}\phi_j\left(\log t, \frac{x}{t}\right) + O(t^{-\delta}),
\end{equation}
for $\delta >1/2$, where $\phi_1$ and $\phi_2$ are the solution to \eqref{eq:ODE} (see also \cite{HLN} 
for the two dimensional case). 
We note that they
considered not only \eqref{eq:2NLS} but also
 more general $N$-coupled systems of cubic nonlinear Schr\"odinger equations, including 
derivatives in the nonlinearity.

In our previous studies \cite{MaSeUr2,KiMaSeUr},
we showed a similar asymptotic results for several examples belonging to the case (3) or (6).
These examples has the special structure;
one equation of the system is the single cubic equation \eqref{eq:NLS} and the other equation is linear with time dependent potential given by the first component.

Our main motivation here is to study the case $\rank \mathscr{A} = 3$. 
We remark that the condition \eqref{eq:WNGC} holds only in the case $\rank \mathscr{A} \le 2$ 
because $\ker \mathscr{A}=\{0\}$ if $\rank \mathscr{A} = 3$.  
%
Recently, the first author~\cite{Ma3} obtained the global existence of the solution for small initial data and its asymptotic 
behavior including the case of $\rank \mathscr{A}=3$. 
He showed that if $\mathscr{A}$ has two pure imaginary eigenvalues, say $\pm ik$ ($k>0$), and if 
$\ker (\mathscr{A}^2 + k^2 E_3) \cap \mathcal{P}_+ \neq \emptyset$, where $E_3$ is the identity matrix of size 3,
 then
the reduced ODE system has the conserved quantity 
which is given as a \emph{quartic quantity} equivalent to 
\begin{equation}
  (|\varphi_1(t)|^2+|\varphi_2(t)|^2)^2. 
\end{equation}
Moreover, the precise asymptotic behavior of the solution is given by explicitly solving the reduced ODE 
for a typical case
(see also \cite{Ma2} for other ``solvable'' cases).

In this paper, we find more systems \eqref{eq:2NLS} which satisfies $\rank \mathscr{A}=3$ and the behavior of small solutions are specified.
More precisely, we prove that \eqref{eq:ODE} has conserved quantities 
which are \emph{not necessarily a polynomial} and described by eigenvalues and eigenvectors of $\mathscr{A}$.
We further find a class of systems of \eqref{eq:2NLS}
for which the global boundedness of solutions to
the corresponding reduced ODE systems \eqref{eq:ODE} is obtained
by utilizing the conservation of these quantities.
%

\subsection{A conserved quantity of \eqref{eq:ODE}}

Before stating our result on \eqref{eq:2NLS}, we state a
 simple result on the reduced ODE system \eqref{eq:ODE}.
This is an immediate consequence of Proposition \ref{prop:ode}.
\begin{theorem}\label{thm:cons}
Let $\mathscr{A}$ be the matrix part of the system \eqref{eq:ODE}.
If $\vec{a}_1 \in \R^3$ and $\vec{a}_2 \in \R^3$ are eigenvectors of $\mathscr{A}$ associated with eigenvalues $\lambda_1$ and $\lambda_2$, respectively, then
\[
	|\mathcal{Q}(\vec{a}_1)|^{\lambda_2} |\mathcal{Q}(\vec{a}_2)|^{-\lambda_1}
\]
is conserved when it is well-defined, where $\mathcal{Q}(\cdot)$ is defined as in \eqref{eq:Q}.
\end{theorem}

By utilizing this conserved quantity, one can prove a global existence and boundedness result for a class of reduced ODE system \eqref{eq:ODE}.
Because it is a main ingredient for our main result, let us introduce here.
To specify the class,
let us now introduce assumptions on the system.
To this end, we make a notation.
For a complex number $\lambda$,
we let
\[
	W(\lambda) = W(\lambda,\mathscr{A})
	= \ker (\mathscr{A}- \lambda E_3) \subset \R^3
\]
be the eigenspace, where $E_3$ is the identity matrix of size 3. Note that $W$ contains nonzero element if and only if $\lambda$ is an eigenvalue of $\mathscr{A}$. 
Our new class of systems are those satisfy one of the following two assumptions:
\begin{assumption}\label{A1}
The matrix $\mathscr{A}$ has two distinct nonzero real
eigenvalues $\lambda_1$ and $\lambda_2$ such that
$W(\lambda_j) \cap \mathcal{P}_+\neq \emptyset$ for $j=1,2$.
\end{assumption}
\begin{assumption}\label{A2}
The matrix $\mathscr{A}$ has three distinct real eigenvalues $\lambda_j$
($j=1,2,3$) such that $\lambda_3\neq0$,
$\frac{\lambda_1}{\lambda_3}<1$, and $\frac{\lambda_2}{\lambda_3}<1$.
Further, 
$W(\lambda_1)\cap \mathcal{P}_+=W(\lambda_2)\cap \mathcal{P}_+=\emptyset$,
$(W(\lambda_1)\oplus W(\lambda_2))\cap \mathcal{P}_+\neq \emptyset$, and $W(\lambda_3) \cap \mathcal{P}_+ \neq \emptyset$.
\end{assumption}
\begin{remark}
\begin{enumerate}
\item 
In Assumptions 1 and 2, 
we assumed $\lambda_1$ and $\lambda_2$ are nonzero. This is because if one of them is zero then the weak null gauge condition \eqref{eq:WNGC} is satisfied.
\item
We assumed $W(\lambda_1)\cap \mathcal{P}_+=W(\lambda_2)\cap \mathcal{P}_+=\emptyset$ in Assumption \ref{A2}
because otherwise Assumption \ref{A1} or 
the weak null gauge condition \eqref{eq:WNGC} is fulfilled.
Note that $\lambda_1\lambda_2=0$ is allowed
in Assumption \ref{A2}.
\item 
We do not need to make any assumptions about $\mathscr{V}$.
\end{enumerate}
\end{remark}

The following result holds.

\begin{theorem}\label{thm:ODE}
We assume Assumptions \ref{A1} or \ref{A2}. 
Then any solution to the reduced ODE system
\eqref{eq:ODE} exists globally in time and satisfies
\[
	\sup_{t\in\R} (|\varphi_1(t)|^2+|\varphi_2(t)|^2)
	\lesssim \inf_{t\in\R} (|\varphi_1(t)|^2+|\varphi_2(t)|^2).
\]
\end{theorem}

\subsection{Main result}

Now we turn to our main topic, i.e., the asymptotic behavior of small solutions to the NLS system \eqref{eq:2NLS}.
We define the Sobolev space $H^1$ by
\[
H^1=H^1(\R)
=\{f\in{{\mathcal S}}'(\R)\ ; \ \|f\|_{H^1}
=\|f\|_{L^2}+\|\pt_x f\|_{L^2}<\infty\},
\]
and the weighted $L^2$ space by 
\[
H^{0,1}=H^{0,1}(\R)
=\{f\in{{\mathcal S}}'(\R)\ ; \ \|f\|_{H^{0,1}}
=\|f\|_{L^2}+\|xf\|_{L^2}<\infty\}.
\]

\begin{theorem}\label{thm:main}
  We assume either Assumption 1 or Assumption 2.
  Then there exists $\epsilon_0 > 0$ such that
  if $\epsilon := \|u_{1,0}\|_{H^{0,1}} + \|u_{2,0}\|_{H^{0,1}}$
  satisfies $\epsilon \le \epsilon_0$, then there exists a unique solution
  $(u_1, u_2) \in (C(\R; L^2(\R)))^2$ to \eqref{eq:2NLS} satisfying
  $(U(-t)u_1, U(-t)u_2) \in (C(\R;H^{0,1}(\R)))^2$.
  Furthermore, there exist
  $\delta > 0$ and $(\psi_1^\pm, \psi_2^\pm) \in C(\R) \cap L^\infty(\R)$
  with $\|\psi_j^\pm\|_{L^\infty} \lesssim \epsilon\ (j=1,2)$ such that
  \begin{equation}\label{eq:asymptest}
    u_j(t,x) = (it)^{-\frac12}e^{\frac{i|x|^2}{2t}}
    A_j^\pm(\sign(t)\log |t|, \tfrac{x}{t})
    + O(\epsilon|t|^{-\frac34 + \frac{1}{2p}+\delta \epsilon^2})
  \end{equation}
  in $L^p(\R)$ as $t \to \infty$ for $p \in [2,\infty]$, where
  $(A_1^\pm(t,\xi), A_2^\pm(t,\xi))$ is a solution to \eqref{eq:ODE}
  subject to the initial condition
  \begin{equation}
    (A_1^\pm(0,\xi), A_2^\pm(0,\xi)) = (\psi_1^\pm(\xi), \psi_2^\pm(\xi)).
  \end{equation}
\end{theorem}

Since our assumption is given in terms of eigenvalue and eigenvector of matrix $\mathscr{A}$,
it is easy to check.
Further, it is also easy to construct examples of $\mathscr{A}$ satisfying the assumption.
\subsection{Examples}
Let us check that the systems \eqref{eq:2NLSa} and \eqref{eq:2NLSb} satisfy our assumptions.

The matrix-vector representation of \eqref{eq:2NLSa} is
\[
	\mathscr{A}_1=\begin{bmatrix}
	\zeta+1 & 0 & -2\zeta \\
	0 & 1 & 0 \\
	\zeta & 0 & - 2\zeta+1
	\end{bmatrix}, \quad
	\mathscr{V}_1 = \begin{bmatrix} 0 \\ \frac{-3\zeta+1}2 \\ 0 \end{bmatrix}.
\]
Since $\zeta\neq0$, it follows that
\[
	W(1;\mathscr{A}_1)= \Span \left\{
	\begin{bmatrix} 2 \\ 0 \\ 1 \end{bmatrix},\,
	\begin{bmatrix} 0 \\ 1 \\ 0 \end{bmatrix}\right\}, \quad
	W(1-\zeta;\mathscr{A}_1)= \Span \left\{
	\begin{bmatrix} 1 \\ 0 \\ 1 \end{bmatrix}\right\}.
\]
Therefore, it satisfies Assumption \ref{A1}.
One sees from Theorem \ref{thm:cons} that the reduced ODE system
\[
	\left\{
  \begin{aligned}
     & i\partial_t \varphi_1  =
     (\zeta+1)|\varphi_1|^2 \varphi_2 + \zeta |\varphi_2|^2\varphi_2,  \\
     & i\partial_t \varphi_2 =
  2\zeta |\varphi_1|^2 \varphi_1+ 2\zeta \varphi_1 |\varphi_2|^2 +\ol{\varphi_1}\varphi_2^2 \end{aligned}
  \right.
\]
has a conserved quantity
\[
	(2|\varphi_1|^2+|\varphi_2|^2)^{\zeta-1} (|\varphi_1|^2+|\varphi_2|^2),
\]
which is \emph{not a power of a polynomial} if $\zeta \not\in \mathbb{Q}$ or $\zeta<1$.
Notice that the quantity is equivalent
to $ (|\varphi_1|^2+|\varphi_2|^2)^{\zeta}$.
Hence, under the assumption $\zeta\neq0$,
the conservation of this quantity implies that the solution is global and uniformly bounded.
Assumption \ref{A1} ensures the existence of a conserved quantity which is equivalent to a power of $|\varphi_1|^2 + |\varphi_2|^2$.
Note that there are two more conserved quantities\footnote{It would be more appropriate to say ``one more conserved quantity'' here since the latter quantity is given in terms of the former quantity and the above quantity, and vice versa.}
\[
	|2\Re (\overline{\varphi_1}\varphi_2)|^{\zeta-1} (|\varphi_1|^2+|\varphi_2|^2), \quad
	|2\Re (\overline{\varphi_1}\varphi_2)| (2|\varphi_1|^2+|\varphi_2|^2)^{-1}.
\]
However, they do not yield the global bound of the solution directly since they are not equivalent to a power of $|\varphi_1|^2 + |\varphi_2|^2$.

The matrix-vector representation of \eqref{eq:2NLSb} is
\[
	\mathscr{A}_2=\begin{bmatrix}
	2 & -6 & 0 \\
	0 & -1 & 0 \\
	0 & -4 & 1
	\end{bmatrix}, \quad
	\mathscr{V}_2 = \begin{bmatrix} 0 \\ 0 \\ 0 \end{bmatrix}.
\]
It is easy to see that $\mathscr{A}_2$ has three distinct eigenvalues $2$, $1$, and $-1$ and
\[
	W(2;\mathscr{A}_2)= \Span \left\{
	\begin{bmatrix} 1 \\ 0 \\ 0 \end{bmatrix}\right\}, \quad
	W(1;\mathscr{A}_2)= \Span \left\{
	\begin{bmatrix} 0 \\ 0 \\ 1 \end{bmatrix}\right\}, \quad
	W(-1;\mathscr{A}_2)= \Span \left\{
	\begin{bmatrix} 2 \\ 1 \\ 2 \end{bmatrix}\right\}.
\]
By Theorem \ref{thm:cons}, one sees that
\[
 \left\{
  \begin{aligned}
     & i\partial_t \varphi_1  =
    6|\varphi_1|^2\varphi_1 +3|\varphi_1|^2\varphi_2 +\varphi_1^2\ol{\varphi_2},  \\
     & i\partial_t \varphi_2  =
    -\varphi_1|\varphi_2|^2 - 4 |\varphi_2|^2 \varphi_2
    \end{aligned}
  \right.
 \]
has two conserved quantities
\[
	h_1:=(2|\varphi_1|^2 + 2\Re (\overline{\varphi_1}\varphi_2)+ 2|\varphi_2|^2)^2(|\varphi_1|^2), \quad h_2:=(2|\varphi_1|^2 + 2\Re (\overline{\varphi_1}\varphi_2)+ 2|\varphi_2|^2)
	(|\varphi_2|^2).
\]
Each of them does not give us 
the global bound of a solution.
However, the desired global bound follows from
a combination of these two quantities.
Indeed, we see that
\[
	h_1^2 + h_2^3 \sim (|\varphi_1|^2+|\varphi_2|^2)^6
\]
is also a conserved quantity.
Assumption \ref{A2} implies that the existence of such a pair of conserved quantities.

\subsection{Standard form of systems satisfying Assumptions \ref{A1} or \ref{A2}}
We next give a standard form of systems
satisfying Assumptions \ref{A1} and \ref{A2}.
The standard forms are explicit examples of the systems.
We let
\[
	\tilde{W}(\lambda) = \tilde{W}(\lambda,\mathscr{A})
	= \bigcup_{k\ge1} \ker ((\mathscr{A}- \lambda E_3)^k) \subset \R^3
\]
be the generalized eigenspace.
Here, we state our result as a transformation of the NLS system \eqref{eq:2NLS}.
However, the transformation of \eqref{eq:ODE} is completely the same.

\begin{theorem}\label{thm:standard1}
Suppose that a system \eqref{eq:2NLS} satisfies Assumption \ref{A1}. Let $(u_1,u_2)$ be the unknown for the system.
Then, there exists $\mathscr{M} \in GL_2(\R)$ such that the matrix part $\mathscr{A}$ 
of the system which the new variable
\[
	\begin{bmatrix}
	v_1 \\ v_2
	\end{bmatrix}
	= \mathscr{M}
	\begin{bmatrix}
	u_1 \\ u_2
	\end{bmatrix}
\]
becomes one of $\mathscr{A}_{1,j}$ for $j=1,2,3$ given as follows:
\begin{itemize}
\item $\mathscr{A}_{1,1}$ is given as
\[
\begin{bmatrix}
	\lambda_1 - (1+\eta_1)(\lambda_1-\lambda_2) & 
	\eta_2(1-\lambda_1) +(1+\eta_1)(\eta_2-\eta_3)(\lambda_1-\lambda_2) 
	& (1+\eta_1)(\lambda_1-\lambda_2) \\
	0 & 1& 0\\
	-\eta_1(\lambda_1-\lambda_2) & \eta_3(1- \lambda_1) +\eta_1(\eta_2-\eta_3)(\lambda_1-\lambda_2)& \lambda_1 +\eta_1(\lambda_ 1- \lambda_2)
	\end{bmatrix}
\]
with $\lambda_1,\lambda_2 \not\in \{0,1\}$, $\lambda_1>\lambda_2$,
$\eta_1>0$, and $\eta_2,\eta_3\in \R$.
$\mathscr{A}_{1,1}$ has three distinct eigenvalues 
$\lambda_1,\lambda_2$, and $1$ and it holds that
\begin{align*}
	W(\lambda_1;\mathscr{A}_{1,1}) 
	&= \Span \{{}^t\!(1,0,1) \}, \\
	W(\lambda_2;\mathscr{A}_{1,1}) 
	&= \Span \{{}^t\!(1+\eta_1,0,\eta_1) \},\\
	W(1;\mathscr{A}_{1,1}) 
	&= \Span \{{}^t\!(\eta_2,1,\eta_3) \}.
\end{align*}
\item $\mathscr{A}_{1,2}$ is given as
\[
	\begin{bmatrix}
	1 - (1+\eta_1)(1-\lambda) & 
	 (1+ \eta_1)(\eta_2 - \eta_3)(1-\lambda) + \eta_4
	& (1+\eta_1)(1-\lambda) \\
	0 & 1 & 0\\
	-\eta_1(1-\lambda) & \eta_1(\eta_2- \eta_3)(1-\lambda)+\eta_4 & 1+\eta_1(1-\lambda)
	\end{bmatrix}
\]
with $\lambda \not\in \{0,1\}$, $\eta_1>0$, and $\eta_2,\eta_3,\eta_4\in \R$.
$\mathscr{A}_{1,2}$ has exactly two distinct eigenvalues
 $1$ and $\lambda$ and it holds that
 \begin{align*}
	W(1;\mathscr{A}_{1,2}) 
	&=\begin{cases}
	 \Span \{{}^t\!(1,0,1) \}, & \eta_4 \neq0,\\
	 \Span \{{}^t\!(1,0,1), {}^t\!(\eta_2,1,\eta_3)\}, & \eta_4=0,
	 \end{cases}
	 \\
	\tilde{W}(1;\mathscr{A}_{1,2}) 
	&= \Span \{{}^t\!(1,0,1), {}^t\!(\eta_2,1,\eta_3) \}, \\
	W(\lambda;\mathscr{A}_{1,2}) 
	&= \Span \{{}^t\!(1+\eta_1,0,\eta_1) \}.
\end{align*}
\item $\mathscr{A}_{1,3}$ is given as
\[
\begin{bmatrix}
	\lambda_1 - (1+\eta_1)(\lambda_1+1) & 
	-\eta_2\lambda_1 +(1+\eta_1)(\eta_2-\eta_3)(\lambda_1+1) 
	& (1+\eta_1)(\lambda_1+1) \\
	0 & 0 & 0\\
	-\eta_1(\lambda_1+1) & -\eta_3 \lambda_1 +\eta_1(\eta_2-\eta_3)(\lambda_1+1)& \lambda_1 +\eta_1(\lambda_ 1+1)
	\end{bmatrix}
\]
with $\lambda_1 \in (-1,1] \setminus\{0\}$, $\eta_1>0$, and $\eta_2,\eta_3\in\R$. 
$\mathscr{A}_{1,3}$ has three distinct eigenvalues $\lambda_1$, $0$, and $-1$ and it holds that
\begin{align*}
	W(\lambda_1;\mathscr{A}_{1,3}) 
	&= \Span \{{}^t\!(1,0,1) \}, \\
	W(0;\mathscr{A}_{1,3}) 
	&= \Span \{{}^t\!(\eta_2,1,\eta_3) \}, \\
	W(-1;\mathscr{A}_{1,3}) 
	&= \Span \{{}^t\!(1+\eta_1,0,\eta_1) \}.
\end{align*}
\end{itemize}
\end{theorem}
Let us next consider the standard form of systems satisfying Assumption \ref{A2}.
\begin{theorem}\label{thm:standard2}
Suppose that a system \eqref{eq:2NLS} satisfies Assumption \ref{A2}. Let $(u_1,u_2)$ be the unknown for the system.
Then, there exists $\mathscr{M} \in GL_2(\R)$ such that the matrix part $\mathscr{A}$ 
of the system which the new variable
\[
	\begin{bmatrix}
	v_1 \\ v_2
	\end{bmatrix}
	= \mathscr{M}
	\begin{bmatrix}
	u_1 \\ u_2
	\end{bmatrix}
\]
solves becomes one of $\mathscr{A}_{2,j}$ for $j=1,2$ given as follows:
\begin{itemize}
\item $\mathscr{A}_{2,1}$ is given as
\[
	\begin{bmatrix}
	\lambda_1 & 
	-\eta(\lambda_1+1) 
	& 0 \\
	0 & -1& 0\\
	0 & -\eta(\lambda_2+1)  & \lambda_2
	\end{bmatrix}
\]
with $\lambda_1 > \lambda_2 >-1$ and $\eta\in \R$ with $|\eta|>1$.
$\mathscr{A}_{2,1}$ has three distinct eigenvalues 
$\lambda_1$, $\lambda_2$, and $-1$ and it holds that
\begin{align*}
	W(\lambda_1;\mathscr{A}_{2,1}) 
	&= \Span \{{}^t\!(1,0,0) \}, \\
	W(\lambda_2;\mathscr{A}_{2,1}) 
	&= \Span \{{}^t\!(0,0,1) \},\\
	W(-1;\mathscr{A}_{2,1}) 
	&= \Span \{{}^t\!(\eta,1,\eta) \}.
\end{align*}
\item $\mathscr{A}_{2,2}$ is given as
\[
	\begin{bmatrix}
	\lambda_1-(1+\eta_1)(\lambda_1-\lambda_2) & 
	(1+\eta_1)(\eta_2+\eta_3)(\lambda_1-\lambda_2) -(\lambda_1+1)\eta_2
	& -(1+\eta_1)(\lambda_1-\lambda_2) \\
	0 & -1& 0\\
	\eta_1(\lambda_1-\lambda_2) & -\eta_1(\eta_2+\eta_3)(\lambda_1-\lambda_2) -(\lambda_1+1)\eta_3  & \eta_1(\lambda_1-\lambda_2)
	+\lambda_1
	\end{bmatrix}
\]
with $\lambda_1,\lambda_2 \in (-1,\infty)$, $\lambda_1\neq\lambda_2$, $\eta_1\ge0$, $\eta_2,\eta_3\in \R$, and $\eta_2\eta_3>1$.
Further, if $\eta_1\neq0$ then $\lambda_1>\lambda_2$.
$\mathscr{A}_{2,2}$ has three distinct eigenvalues $\lambda_1$, $\lambda_2$, and $-1$ and it holds that
\begin{align*}
	W(\lambda_1;\mathscr{A}_{1,3}) 
	&= \Span \{{}^t\!(1,0,-1) \}, \\
	W(\lambda_2;\mathscr{A}_{1,3}) 
	&= \Span \{{}^t\!(1+\eta_1,0,-\eta_1) \}, \\
	W(-1;\mathscr{A}_{1,3}) 
	&= \Span \{{}^t\!(\eta_2,1,\eta_3) \}.
\end{align*}
\end{itemize}
\end{theorem}

\smallskip

The rest of the paper is organized as follows. 
In Sections 2 and 3, we prove Theorem \ref{thm:main}. 
In particular, in Section 2, we show the global existence of 
solution to \eqref{eq:2NLS} and in Section 3, we derive the 
asymptotic formula (\ref{eq:asymptest}). Finally in Section 4, 
we prove Theorems \ref{thm:standard1} and \ref{thm:standard2}.

\section{Global existence of the solution}

In this section we prove the global existence of 
solution to \eqref{eq:2NLS}. 
We will denote by $U(t)$ the free Schr\"odinger group, 
namely,
\begin{equation}
  (U(t)f)(x) = (e^{\frac{it}{2}\partial_x^2}f)(x)
  =\left(\frac{1}{2\pi it}\right)^{\frac12}
  \int_{\R}e^{\frac{i|x-y|^2}{2t}}f(y)dy.
\end{equation}
The following factorization formula is well-known:
\begin{equation}
  U(t) = M(t)D(t)\mathcal{F}M(t),
\end{equation}
where 
\begin{align*}
(M(t)f)(x) &= e^{\frac{i|x|^2}{2t}}f(x),\qquad(\text{Multiplication})\\
(D(t)f)(x) &= (it)^{-\frac12}f\left(\frac{x}{t}\right),
\qquad(\text{Dilation})
\end{align*}
and $\mathcal{F}$ is the Fourier transform on $\R$.
The generator of  the Galilean transformation is given as 
$J(t)=x+it\partial_x$. 
Note that the operators $J(t)$ and 
$i\partial_t +(1/2)\partial^2_x$ commute. 
The following identities are useful:
\[
	J(t) = U(t) x U(-t) = M(t) it \partial_x M(t)^{-1},
\]
where the second identity is valid for $t\neq0$.

By the standard theory of local well-posedness for 
the nonlinear Schr\"{o}dinger equation 
(see \cite{Caz} for instance), 
we have the following.

\begin{proposition}\label{prop:local}
  There exists $\widetilde{\epsilon}_0 > 0$ such that if
  $\epsilon := \|u_{1,0}\|_{H^{0,1}} + \|u_{2,0}\|_{H^{0,1}}$ satisfies
  $\epsilon \le \widetilde{\epsilon}_0$, then there exists
  a unique solution
  $(u_1, u_2) \in (C([-1, 1];L^2(\R)))^2$ to \eqref{eq:2NLS} such that
  \[
 ( U(-t)u_1, U(-t)u_2) \in (C([-1,1];H^{0,1}(\R)))^2
  \]
  and
  \begin{equation}
    \max_{t \in [-1,1]}\left(
    \|u_1(t)\|_{L^2} + \|J(t)u_1(t)\|_{L^2}
    + \|u_2(t)\|_{L^2} + \|J(t)u_2(t)\|_{L^2}
    \right) \le 2\epsilon.
  \end{equation}
\end{proposition}
We shall show the existence of global solution by 
proving that $\|U(-t)u_1\|_{H^{0,1}}+\|U(-t)u_2\|_{H^{0,1}}$ is finite for all $t>0$.
To this end,
 we introduce a new variable
  $w_j(t,\xi) := \mathcal{F}[U(-t)u_j(t,\cdot)](\xi)$.
  Note that $U(-t)u_j(t) \in H^{0,1}$ is equivalent to $w_j(t) \in H^1$.
  Applying $\mathcal{F}U(-t)$ to the both sides of \eqref{eq:2NLS}, we have
  \begin{equation}\label{eq:wsys}
    i\partial_t w_j(t) = \frac{1}{t}F_j(w_1, w_2) + r_j
  \end{equation}
  for $j = 1,2$ and $t \neq 0$, where
  \begin{equation}\label{eq:rj}
    r_j := \textrm{I}_j + \textrm{II}_j
  \end{equation}
  with
  \begin{equation}\label{eq:Ijdef}
    \textrm{I}_j := \left(U\left(\tfrac{1}{t}\right)-1\right)D(t)^{-1}M(t)^{-1}F_j(u_1(t), u_2(t))
  \end{equation}
  and
  \begin{equation}\label{eq:IIjdef}
    \textrm{II}_j :=
    \frac{1}{t}\left\{(F_j\left(U\left(-\tfrac{1}{t}\right)w_1(t),U\left(-\tfrac{1}{t}\right)w_2(t)\right)
    - F_j(w_1(t), w_2(t)))\right\}.
\end{equation}
From the unitarity of the Fourier transform 
and $U(t)$ in $L^2(\R)$ and $J(t)$, we have
  \begin{equation}\label{eq:L2}
    \|w_j(t)\|_{L^2} = \|u_j(t)\|_{L^2}, \quad 
    \|\partial_{\xi} w_j(t)\|_{L^2} = \|J(t)u_j(t)\|_{L^2}.
  \end{equation}
  Since $u_j(t) = M(t)D(t)U(-\frac{1}{t})w_j(t)$, it holds that 
  for any $t\neq0$, 
  \begin{equation}\label{eq:Linfty}
    \|U\left(-\tfrac{1}{t}\right)w_j(t)\|_{L^\infty} 
    = t^{\frac{1}{2}}\|u_j(t)\|_{L^2}.
  \end{equation}


For $\delta > 0$ and $T > 1$, we define
\begin{equation}
  X_T := \sup_{t \in [1,T]}t^{-\delta\epsilon^2}
  \sum_{j = 1}^2(\|u_j(t)\|_{L^2} + \|J(t)u_j(t)\|_{L^2})
\end{equation}
and
\begin{equation}
  Y_T := \sup_{t \in [1,T]}t^{\frac12}\sum_{j=1}^2\|u_j(t)\|_{L^\infty}.
\end{equation}
Note that by the Gagliardo-Nirenberg inequality and (\ref{eq:L2}), 
we have for any $t\in [1,T]$, 
\begin{equation}\label{eq:wjinfty}
\|w_j(t)\|_{L^\infty} 
\lesssim
\|w_j(t)\|_{L^2}^{\frac12}
\|\pt_{\xi}w_j(t)\|_{L^2}^{\frac12}
\lesssim 
t^{\delta\varepsilon^2}X_T.
\end{equation}
The following gives us the global existence of a solution to \eqref{eq:2NLS}.

\begin{proposition}\label{prop:global}
Let $\widetilde{\epsilon_0}$ be given by Proposition 
\ref{prop:local}. 
Then there exist $\delta > 0$, $C > 0$, and $\epsilon_0 \in (0,\widetilde{\epsilon_0}]$ 
such that if $\epsilon= \|u_{1,0}\|_{H^{0,1}} + \|u_{2,0}\|_{H^{0,1}}$ 
satisfies $\epsilon\le\epsilon_0$, then the unique solution to \eqref{eq:2NLS} 
given in Proposition \ref{prop:local} exists globally 
in time for positive time direction and obeys the bound
  \begin{equation}
    \sup_{T \ge 1}\left(X_T + Y_T\right) \le C\epsilon.
  \end{equation}
\end{proposition}
\begin{proof}
Since $X_T$ and $Y_T$ are continuous in $T$ on $[1,\infty)$,
 it suffices to show that there exist $\delta > 0$, $C_0 > 0$, $\widetilde{C}_0 > 0$, and $\epsilon_0 > 0$ such that if
  $\epsilon \in (0,\epsilon_0]$ and 
  \begin{equation}\label{eq:boot1}
    X_T \le C_0\epsilon, \quad Y_T \le \widetilde{C}_0\epsilon
  \end{equation}
  for some $T \ge 1$ then it holds that
  \begin{equation}\label{eq:boot2}
    X_T \le \tfrac{1}{2}C_0\epsilon, \quad Y_T \le \tfrac{1}{2}\widetilde{C}_0\epsilon.
\end{equation}
We divide the proof into four steps.

{\bf Step 1}. We estimate $X_T$.
Let $\delta > 0$, $C_0 > 0$, $\tilde{C}_0$, and $\epsilon_0 > 0$ to be chosen later. Suppose that \eqref{eq:boot1} is true for some $T \ge 1$.
  By the Duhamel formula for \eqref{eq:2NLS}, we see that
  \begin{align*}
    \sum_{j=1}^2\|u_j(t)\|_{L^2}
     & \le C_1\epsilon + 
     C_1\int_1^t\left(\sum_{j=1}^2\|u_j(\tau)
     \|_{L^\infty}\right)^2\sum_{j=1}^2\|u_j(\tau)\|_{L^2}d\tau \\
     & \le C_1\epsilon + C_1X_TY_T^2\int_1^t \tau^{\delta\epsilon^2-1}
     d\tau                                            \\
     & \le C_1\epsilon + \delta^{-1}\widetilde{C}_0^2C_0C_1\epsilon t^{\delta\epsilon^2}
  \end{align*}
  for $t \ge 1$.
  In the same manner, we obtain the estimate for $J(t)u_j(t)$:
  \begin{equation}
    \sum_{j=1}^2\|J(t)u_j(t)\|_{L^2} \le C_1\epsilon 
    +\delta^{-1}\widetilde{C}_0^2C_0C_1 \epsilon t^{\delta\epsilon^2}
  \end{equation}
  for $t \ge 1$.
  Hence, it holds that
  \begin{equation}
    X_T \le C_1(1+\delta^{-1}C_0\widetilde{C}_0^2)\epsilon.
    \label{eq:XT1}
  \end{equation}
  
  {\bf Step 2}. We estimate $Y_T$. 
  By using the Gagliardo-Nirenberg inequality and \eqref{eq:L2}, we have
  \begin{align}\label{eq:w1}
    \|(U(-\tfrac{1}{t})-1)w_j(t)\|_{L^\infty} & \lesssim \|(U(-\tfrac{1}{t})-1)w_j(t)\|_{L^2}^\frac12\|\partial_x(U(-\tfrac{1}{t})-1)w_j(t)\|_{L^2}^\frac12 \\ \nonumber           & \lesssim t^{-\frac14}\|\partial_x w_j(t)\|_{L^2} \\ \nonumber
& = t^{-\frac14} \|J(t)u_j(t)\|_{L^2}.
  \end{align}
  Hence, for $t \in [1,T]$, it holds that
  \begin{align*}
    \|u_j(t)\|_{L^\infty}
     & = \|U(t)U(-t)u_j(t)\|_{L^\infty}                              \\
     & \le \|M(t)D(t)w_j(t)\|_{L^\infty}
    + \|M(t)D(t)\mathcal{F}(M(t)-1)\mathcal{F}^{-1}w_j(t)\|_{L^\infty} 
    \nonumber\\
     & \le Ct^{-\frac{1}{2}}\|w_j(t)\|_{L^\infty}
    + Ct^{-\frac12}\|(U(-\tfrac{1}{t})-1)w_j(t)\|_{L^\infty}  \nonumber \\
     & \le Ct^{-\frac{1}{2}}\|w_j(t)\|_{L^\infty}
    + Ct^{-\frac34}\|J(t)u_j(t)\|_{L^2}.\nonumber
  \end{align*}
This yields 
  \begin{equation}
    Y_T \le C\sum_{j=1}^2\|w_j(t)\|_{L^\infty} + CX_T.\label{yt}  
  \end{equation}

In view of \eqref{yt}, to estimate $Y_T$, 
it suffices to estimate $w_j$ in $L^\infty$.
We shall do so in the next two steps by using the integral form of \eqref{eq:wsys}. 
 In this step, we estimate the remainder term $r_j$ given in 
 \eqref{eq:rj}-\eqref{eq:IIjdef}.
  For the estimate of $\textrm{I}_j$ given in 
  \eqref{eq:Ijdef}, it follows from the Gagliardo-Nirenberg inequality that
  \begin{align*}
    \|\textrm{I}_j\|_{L^\infty} & \lesssim
    \|\textrm{I}_j\|_{L^2}^\frac12 \|\partial_x\textrm{I}_j\|_{L^2}^\frac12                                        \\
                                & \lesssim t^{-\frac14}\|\partial_xD(t)^{-1}M(t)^{-1}F_j(u_1(t), u_1(t))\|_{L^2}   \\
                                & \lesssim t^{-\frac14}\|J(t)F_j(u_1(t), u_2(t))\|_{L^2}                           \\
                                & \lesssim t^{-\frac54}\left(\sum_{j=1}^2t^{\frac12}\|u_j(t)\|_{L^\infty}\right)^2
    \sum_{j=1}^2\|J(t)u_j(t)\|_{L^2}                                                                               \\
                                & \lesssim t^{-\frac54+\delta\epsilon^2}X_TY_T^2.
  \end{align*}
As for $\textrm{II}_j$ given in \eqref{eq:IIjdef}, the estimates 
\eqref{eq:Linfty}, \eqref{eq:wjinfty} and \eqref{eq:w1} yield
  \begin{align*}
    \|\textrm{II}_j\|_{L^\infty} & \lesssim t^{-1}
    \left(\sum_{j=1}^2\left\|U(-\tfrac{1}{t})w_j(t)\right\|_{L^\infty}^2 + \sum_{j=1}^2\|w_j(t)\|_{L^\infty}^2\right)
    \sum_{j=1}^2\left\|(U(-\tfrac{1}{t})-1)w_j(t)\right\|_{L^\infty}                          \\
                                 & \lesssim t^{-1}
    \left(\sum_{j=1}^2(t^{\frac12}\left\|u_j(t)\right\|_{L^\infty})^2 
    + \sum_{j=1}^2\|w_j(t)\|_{L^\infty}^2\right)
    \sum_{j=1}^2t^{-\frac14}\left\|\partial_x w_j(t)\right\|_{L^2}                            \\
                                 & \lesssim t^{-\frac54+3\delta\epsilon^2}(X_T^2 + Y_T^2)X_T.
  \end{align*}
  Thus, we end up with
  \begin{equation}\label{eq:rjest}
  \|r_j(t)\|_{L^\infty} \lesssim
  t^{-\frac54+3\delta\epsilon^2}(X_T^2 + Y_T^2)X_T
  \end{equation}
  in view of \eqref{eq:rj}.
  
  {\bf Step 3}. We complete the estimates of $Y_T$
under Assumption \ref{A1}.
One may suppose that $\lambda_1>\lambda_2$ without loss of generality.
Let 
  \begin{equation*}
    \rho_1 = |w_1|^2,\quad \rho_2 = |w_2|^2, \quad \mathcal{R} = 2\Re(\ol{w_1}w_2),\quad  \mathcal{I} = 2\Im(\ol{w_1}w_2) 
  \end{equation*}
  and for $\vec{a}={}^t\!(a_1,a_2,a_3) \in \R^3$,
  \begin{equation}
  \mathcal{Q}(\vec a) := \begin{bmatrix}
      \rho_1 & \mathcal{R} & \rho_2
    \end{bmatrix} \vec{a}= \begin{bmatrix}
    \overline{w_1} &
    \overline{w_2}
    \end{bmatrix}
    \begin{bmatrix}
      a_1 & a_2 \\ a_2 & a_3
    \end{bmatrix}
   \begin{bmatrix}
   w_1 \\ w_2
	\end{bmatrix}.\label{QA}
  \end{equation}
Then, we see from \eqref{eq:wsys} and Proposition \ref{prop:ode} that
  \begin{equation}\label{eq:rhosysgen}
    \frac{d}{dt}
   \mathcal{Q}(\vec a)  = 2t^{-1}\mathcal{I} 
    \mathcal{Q}(\mathscr{A}\vec a)
  + \vec{r}\cdot \vec{a}
    \end{equation}
    for any $\vec{a} \in \R^3$, where $\vec{r}=\vec{r}(t)$ is given by
    \begin{equation}\label{eq:rvec}
        \!\vec{r} =
    {}^t\begin{bmatrix}
      2\Im(\ol{w_1}r_1)                     &
      2\Im(\ol{w_1r_2}) + 2\Im(\ol{w_2}r_1) &
      2\Im(\ol{w_2}r_2)
    \end{bmatrix}.
  \end{equation}
Pick 
$
	\vec{p}_j = {}^t\!(
	p_{j,1} ,
	p_{j,2} ,
	p_{j,3} 
	)
	\in W(\lambda_j) \cap \mathcal{P}_+$ and define
\[
	Q_j := \mathcal{Q}(\vec{p}_j),\qquad
	R_j = \vec{r} \cdot \vec{p}_j
\]
for $j=1,2$.
Without loss of generality, we may suppose that $p_{j,1}>0$ so that
   \begin{equation}\label{eq:rhopmequiv}
   Q_1 \sim |w_1|^2 + |w_2|^2 \sim Q_2.
   \end{equation}
Then, one sees from \eqref{eq:rhosysgen} and the fact that $\vec{p}_j$ is eigenvector 
of $\mathscr{A}$ associated with $\lambda_j$ for $j=1,2$, respectively, that
\begin{equation}\label{eq:tildeode}
	\frac{d}{dt} Q_j = 2\lambda_j t^{-1} \mathcal{I} Q_j + R_j.
\end{equation}
  Let $m$ be a real number satisfying $m > 1/(\lambda_1-\lambda_2)$.
  Then, by \eqref{eq:tildeode},
  we have
  \begin{align*}
    \pt_t(Q_1^{-m\lambda_2}Q_2^{m\lambda_1})
     & = -m\lambda_2Q_1^{-m\lambda_2-1}Q_2^{m\lambda_1}  
     \pt_tQ_1
    + m\lambda_1Q_1^{-m\lambda_2}Q_2^{m\lambda_1-1}
    \pt_tQ_2     \\
     & = -m\lambda_2Q_1^{-m\lambda_2-1}Q_2^{m\lambda_1}
    (2t^{-1}\mathcal{I}\lambda_1Q_1 + R_1)\\
    &\quad+ m\lambda_1Q_1^{-m\lambda_2}Q_2^{m\lambda_1-1}(2t^{-1}\mathcal{I}\lambda_2Q_2 + R_2) \\
     & = -m\lambda_2 Q_1^{-m\lambda_2-1}Q_2^{m\lambda_1}R_1
    + m\lambda_1Q_1^{-m\lambda_2}Q_2^{m\lambda_1-1}R_2
  \end{align*}
  if $|w_1|^2+|w_2|^2>0$. This identity is valid also when $|w_1|^2+|w_2|^2=0$ in view of \eqref{eq:rhopmequiv} and $m(\lambda_1-\lambda_2)>1$.
  Integrating this equation on $[1,t]$, we see that
  \begin{align}\label{eq:inteq}
    Q_1(t)^{-m\lambda_2}Q_2(t)^{m\lambda_1}
     & = Q_1(1)^{-m\lambda_2}Q_2(1)^{m\lambda_1}                       \\ \nonumber
     & \quad + m\int_1^t(Q_1^{-m\lambda_2-1}Q_2^{m\lambda_1-1}
    (-\lambda_2Q_2R_1 + \lambda_1Q_1R_2))(\tau) d\tau.
\end{align}
  Substituting 
  \eqref{eq:wjinfty},  \eqref{eq:rjest}, and \eqref{eq:rhopmequiv} into \eqref{eq:inteq}, we obtain
  \begin{align*}
     & (|w_1(t)|^{2} + |w_2(t)|^{2})^{m(\lambda_1-\lambda_2)}   \\
     & \lesssim (|w_1(1)|^{2} + |w_2(t)|^{2})^{m(\lambda_1-\lambda_2)} \\
     &\qquad + 
    \int_1^t \(\|w_1(\tau)\|_{L^\infty}^2+\|w_2(\tau)\|_{L^\infty}^2\)^{m(\lambda_1-\lambda_2)-\frac12} (\| r_1(\tau)\|_{L^\infty}+\| r_2(\tau)\|_{L^\infty})
    d\tau                                                 \\
    & \lesssim \epsilon^{2m(\lambda_1-\lambda_2)} + (X_T^2 + Y_T^2)^{m(\lambda_1-\lambda_2)+1}
    \int_1^t\tau^{-\frac{5}{4}+2\{m(\lambda_1-\lambda_2)+1\}\delta\epsilon_0^2}d\tau                                                 \\
     & \lesssim \epsilon^{2m(\lambda_1-\lambda_2)} + (X_T^2 + Y_T^2)^{m(\lambda_1-\lambda_2)+1}
  \end{align*}
  as long as $-\frac{5}{4}+2\{m(\lambda_1-\lambda_2)+1\}\delta\epsilon_0^2 \le -\frac98$.
  Here, we note that 
\[
\|w_j(1)\|_{L^\infty} \lesssim\|u_j(1)\|_{L^2}^\frac12 
\|J(1)u_j(1)\|_{L^2}^ \frac12 \lesssim\epsilon
\]
follows from Proposition \ref{prop:local}.
This yields that
\begin{equation*}
    \sum_{j=1}^2\|w_j(t)\|_{L^\infty} \lesssim \epsilon + (X_T^2 + Y_T^2)^{\frac{1}{2}+\frac{1}{2m(\lambda_1-\lambda_2)}}.
  \end{equation*}
 Plugging this estimate with \eqref{yt}, we obtain
  \begin{equation*}
    Y_T \le C_2\epsilon + C_2(X_T^2 + Y_T^2)^{\frac{1}{2}+\frac{1}{2m(\lambda_1-\lambda_2)}} + C_2X_T
  \end{equation*}
  with some constant $C_2>0$.
  Substituting 
  \eqref{eq:boot1} into the above inequality, 
  we have
  \begin{align}
Y_T & \le C_2\epsilon + C_2(C_0^2 + \tilde{C}_0^2)^{\frac{1}{2}+\frac{1}{2m(\lambda_1-\lambda_2)}}
    \epsilon^{1+\frac{1}{m(\lambda_1-\lambda_2)}} + C_2C_0\epsilon        
                  \label{eq:YT1}    \\
                  & \le C_2(1+C_0)\epsilon
    + C_2(C_0^2 + \tilde{C}_0^2)^{\frac{1}{2}+\frac{1}{2m(\lambda_1-\lambda_2)}}\epsilon_0^{\frac{1}{m(\lambda_1-\lambda_2)}}\epsilon.
    \nonumber
  \end{align}
  
So far, we obtained (\ref{eq:XT1}) and \eqref{eq:YT1}, i.e.,
\begin{equation*}
    \left\{
    \begin{aligned}
      X_T & \le C_1(1+\delta^{-1}C_0\widetilde{C}_0^2)\epsilon, \\
      Y_T & \le C_2(1+C_0)\epsilon
      + C_2(C_0^2 + \tilde{C}_0^2)^{\frac{1}{2}+\frac{1}{2m(\lambda_1-\lambda_2)}}\epsilon_0^{\frac{1}{m(\lambda_1-\lambda_2)}}\epsilon.
    \end{aligned}
    \right.
 \end{equation*}
To conclude \eqref{eq:boot2}, we choose constants $\delta > 0$, $C_0$, $\tilde{C}_0$, and $\epsilon_0 > 0$ suitable.
  Firstly, we take the constant $C_0$ such that
  \begin{equation}\nonumber
    4C_1 \le C_0.
  \end{equation}
  Secondly  we choose the constant $\tilde{C}_0$ such that
  \begin{equation}\nonumber
    4C_2(1+C_0) \le \tilde{C}_0.
  \end{equation}
  Thirdly, we take the constant $\delta$ such that
  \begin{equation}\nonumber
    4C_1\tilde{C}_0^2 \le \delta,
  \end{equation}
  and finally, we take the constant $\epsilon_0$ such that
  \begin{equation}\nonumber
    4C_2(C_0^2 + \tilde{C}_0^2)^{\frac{1}{2}+\frac{1}{2m(\lambda_1-\lambda_2)}}\epsilon_0^{\frac{1}{m(\lambda_1-\lambda_2)}}
    \le \tilde{C}_0.
  \end{equation}
  Then it holds that
  \begin{equation}\nonumber
    X_T \le \tfrac14C_0\epsilon + \tfrac14C_0\epsilon = \tfrac12C_0\epsilon
  \end{equation}
  and
  \begin{equation}\nonumber
    Y_T \le \tfrac14\tilde{C}_0\epsilon + \tfrac14\tilde{C}_0\epsilon = \tfrac12 \tilde{C}_0\epsilon.
  \end{equation}
  Hence, we obtain \eqref{eq:boot2}.
  In the above argument, by choosing $\epsilon_0$ even smaller if necessary, we also assume that
  $-\frac{5}{4}+2\{m(\lambda_1-\lambda_2)+1\}\delta\epsilon_0^2 \le -\frac98$.

 {\bf Step 4}. We next derive the $L^\infty$ estimate of the solution 
 to (\ref{eq:2NLS}) under Assumption \ref{A2}. 
  Let 
	$\mathcal{Q}(\vec{a})$ be given by (\ref{QA}).
  Pick
 \[
	\vec{p}_j = {}^t\!(
	p_{j,1} ,
	p_{j,2} ,
	p_{j,3} 
	)
	\in W(\lambda_j) 
\]
so that 
$\vec{p}_1 + \vec{p}_2 \in \mathcal{P}_+$, $\vec{p}_{1}\neq0$, $\vec{p}_2\neq0$,
$\vec{p}_3 \in \mathcal{P}_+$, and $p_{3,1}>0$.
It is possible by assumption.
We introduce
\[
	Q_{+,j} := \mathcal{Q}(\vec{p}_j), \quad
	R_{+,j}:= \vec{r} \cdot \vec{p}_{j}
\]
for $j=1,2$ and
\[
	Q_{-} := \mathcal{Q}(\vec{p}_3), \quad R_{-}:= \vec{r} \cdot \vec{p}_3,
\]
where $\vec{r}$ is as in \eqref{eq:rvec}.
Then, one has
\begin{equation}\label{eq:tildeode2}
	\frac{d}{dt} Q_{+,j} = 2t^{-1} \mathcal{I} (\lambda_j  Q_{+,j}) + R_{+,j}
\end{equation}
for $j=1,2$ and
\begin{equation}\label{eq:tildeode3}
	\frac{d}{dt} Q_{-} = 2t^{-1} \mathcal{I} (\lambda_3  Q_{-}) + R_{-}.
\end{equation}
As in the previous step, since
$\vec{p}_3 \in \mathcal{P}_+$ and $p_{3,1}>0$, one sees that
 \begin{equation}\label{eq:rho2}   D_-^{-1} (|w_1|^2 + |w_2|^2) \le Q_- \le D_-(|w_1|^2 + |w_2|^2)
  \end{equation}
  for some constant $D_-\ge1$.
  Further,
by assumption $\vec{p}_1 + \vec{p}_2 \in \mathcal{P}_+$,
one has
\[
	|Q_{+,1} + Q_{+,2}| \gtrsim |w_1|^2 + |w_2|^2.
\]
Further, by an elementary computation,
\[
	|Q_{+,1} + Q_{+,2}|
	\le |Q_{+,1}| + |Q_{+,2}| \lesssim |w_1|^2 + |w_2|^2.
\]
Hence, there exists a constant $D_+\ge 1$ such that
\begin{equation}\label{eq:rho1r} 
 D_+^{-1}(|w_1|^2 + |w_2|^2) \le |Q_{+,1}| + |Q_{+,2}| \le D_+(|w_1|^2 + |w_2|^2).
\end{equation}
In particular, we see from \eqref{eq:rho2} and \eqref{eq:rho1r} that
\begin{equation}\label{eq:rho3}
	\max (|Q_{+,1}| , |Q_{+,2}|) \le D_+D_- Q_{-}.
\end{equation}

For $j=1,2$,
we define
\[
	\theta_j = (1-\tfrac{\lambda_j}{\lambda_3})^{-1} >0.
\]
We may suppose that $\theta_1>\theta_2$ without loss of generality.
We note that the elementary relation
\begin{equation}\label{eq:ltrelation}
	\lambda_j \theta_j + \lambda_3 (1-\theta_j) = 0
\end{equation}
holds for $j=1,2$.
Fix $m >1$ so that $m\theta_2>1$.
For $j=1,2$, by \eqref{eq:tildeode2}, \eqref{eq:tildeode3}, 
we have
  \begin{align*}
     & \pt_t(|Q_{+,j}|^{m\theta_j}Q_-^{m (1-\theta_j)})\\
    &= m\theta_j|Q_{+,j}|^{m\theta_j-2}
    Q_{+,j}\pt_tQ_{+,j}Q_-^{m(1- \theta_j)}
    + m(1-\theta_j)|Q_{+,j}|^{m\theta_j}
    Q_-^{m(1-\theta_j)-1}\pt_tQ_-\\
    &= m\theta_j|Q_{+,j}|^{m\theta_j-2}
    Q_{+,j}(2\lambda_j t^{-1}\mathcal{I}Q_{+,j} + R_{+,j})Q_-^{m(1-\theta_j)} \\
    &\quad
    + m(1-\theta_j) |Q_{+,j}|^{m\theta_j} Q_-^{m(1-\theta_j)-1}(2\lambda_3t^{-1}\mathcal{I}Q_- + R_-)\\
	&=  m\theta_j|Q_{+,j}|^{m\theta_j-2}
    Q_{+,j}Q_-^{m(1-\theta_j)}R_{+,j} + m(1-\theta_j) |Q_{+,1}|^{m\theta_j} Q_-^{m(1-\theta_j)-1} R_-,
  \end{align*}
  where we have used \eqref{eq:ltrelation} to obtain the final line.
 Note that this identity is valid also when $w_1=w_2=0$ since the right hand side becomes zero in view of \eqref{eq:rho3},
 $m\theta_j>1$, and $m>1$.
  Integrating this equation on $[1,t]$ and using the equivalence relations \eqref{eq:rho2} and \eqref{eq:rho3}, and 
the estimates \eqref{eq:wjinfty} and \eqref{eq:rjest}, one obtains
\begin{equation}\label{eq:rhopjest}
\begin{aligned}
\lefteqn{|Q_{+,j}(t)|^{m\theta_j}Q_-(t)^{m (1-\theta_j)}}\qquad\qquad\\
	&\lesssim {}|Q_{+,j}(1)|^{m\theta_j}Q_-^{m(1- \theta_j)}(1)\\
	&+ \int_1^t \(\|w_1(\tau)\|_{L^\infty}^2+\|w_2(\tau)\|_{L^\infty}^2\)^{m-\frac12} (\| r_{1}(\tau)\|_{L^\infty}+\| r_2(\tau)\|_{L^\infty})
    d\tau                                                \\
     &\lesssim \eps^{2m}  +  (X_T^2+Y_T^2)^{m+1}   \int_1^t\tau^{-\frac{5}{4}+2(m+1)\delta\epsilon_0^2}d\tau                                                 \\
     &\lesssim \eps^{2m} + C (X_T^2+Y_T^2)^{m+1}
\end{aligned}
\end{equation}
as long as $-\frac{5}{4}+2(m+1)\delta\epsilon_0^2 \le -\frac98$.
Further, recalling $m\theta_1>m\theta_2>1$, we see 
 from \eqref{eq:rho3} and \eqref{eq:rho1r} that 
  \begin{align*}
    & |Q_{+,1}(t)|^{m\theta_1}Q_-(t)^{m(1-\theta_1)}
    +  |Q_{+,2}(t)|^{m\theta_2}Q_-(t)^{m(1-\theta_2)}\\
    &\ge |Q_{+,1}(t)|^{m\theta_1}Q_-(t)^{m(1-\theta_1)}
    + |Q_{+,2}(t)|^{m\theta_2}\((D_+D_-)^{-1}|Q_{+,2}(t)|\)^{m (\theta_1-\theta_2)}Q_-(t)^{m(1-\theta_1)}\\
    &\ge (D_+D_-)^{-m(\theta_1-\theta_2)} (|Q_{+,1}(t)|^{m\theta_1}+ |Q_{+,2}(t)|^{m\theta_1})Q_-(t)^{m(1-\theta_1)}\\
    &\ge 2^{1-m\theta_1} (D_+D_-)^{-m(\theta_1-\theta_2)} (|Q_{+,1}(t)|+ |Q_{+,2}(t)|)^{m\theta_1}Q_-(t)^{m(1-\theta_1)}\\
    &\ge 2^{1-m\theta_1}  (D_+D_-)^{-m(\theta_1-\theta_2)}  D_+^{-m\theta_1}D_-^{-m(1-\theta_1)}
   (|w_1(t)|^2+ |w_2(t)|^2)^{m}.
  \end{align*}
Thus, we obtain
\[
	(|w_1(t)|^2+ |w_2(t)|^2)^{m}
	\lesssim \eps^{2m} + C (X_T^2+Y_T^2)^{m+1},
\]
which shows
\begin{equation}\label{eq:YT2}
	Y_T	\le C_3 \eps + C_3 (X_T^2+Y_T^2)^{\frac12+ \frac1{2m}}. 
\end{equation}
Arguing as in Step 3, we are able to choose constants $C_0$ and $\tilde{C}_0$ suitably and
obtain \eqref{eq:boot2}.
%
\end{proof}

\begin{remark}
The global existence result for the ODE system (Theorem \ref{thm:ODE}) follows 
from the arguments used in Steps 3 and 4 of the proof.
\end{remark}

We conclude this section with a global-in-time estimate on the error terms $r_j$ given in \eqref{eq:rj}.
\begin{corollary}\label{cor:rj}
  Let $r_j$ 
  be as in \eqref{eq:rj}.
  Let $\delta$ and $\epsilon_0$ be as in Proposition \ref{prop:global}.
  If $\epsilon= \|u_{0,1}\|_{H^{0,1}} + \|u_{0,2}\|_{H^{0,1}}
\le\epsilon_0$, then it holds that
  \begin{equation}\label{eq:estrj}
    \|r_j(t)\|_{L^2} \lesssim \epsilon^3t^{-\frac{3}{2}+3\delta\epsilon^2}, \quad
    \|r_j(t)\|_{L^\infty} \lesssim \epsilon^3t^{-\frac54+3\delta\epsilon^2}
  \end{equation}
  for $t > 1$.
\end{corollary}

\begin{proof}
By Proposition \ref{prop:global},
\eqref{eq:boot1} is true for all $T>0$.
Hence, the $L^\infty$-estimate for $r_j$ 
follows from \eqref{eq:rjest}.
Let us prove $L^2$-estimate.
Arguing as in Step 2 of the proof of Proposition \ref{prop:global},
one has
  \begin{align*}
    \|\textrm{I}_j\|_{L^2} & \lesssim t^{-\frac12}
\|\partial_xD(t)^{-1}M(t)^{-1}F_j(u_1(t), u_1(t))\|_{L^2}   \\
                                  & \lesssim t^{-\frac32+\delta\epsilon^2}X_TY_T^2.
  \end{align*}
Similarly,
  \begin{align*}
    \|\textrm{II}_j\|_{L^2} & \lesssim t^{-1}
    \left(\sum_{j=1}^2\left\|U(-\tfrac{1}{t})w_j(t)\right\|_{L^\infty}^2 + \sum_{j=1}^2\|w_j(t)\|_{L^\infty}^2\right)
    \sum_{j=1}^2\left\|(U(-\tfrac{1}{t})-1)w_j(t)\right\|_{L^2}                          \\
                                 & \lesssim t^{-1}
    \left(\sum_{j=1}^2(t^{\frac12}\left\|u_j(t)\right\|_{L^\infty})^2 + 
    \sum_{j=1}^2\|w_j(t)\|_{L^\infty}^2\right)
    \sum_{j=1}^2t^{-\frac12}\left\|\partial_x w_j(t)\right\|_{L^2}                            \\
                                 & \lesssim t^{-\frac32+3\delta\epsilon^2}(X_T^2 + Y_T^2)X_T.
\end{align*}
Thus we have \eqref{eq:estrj} for $L^2$ norm of $r_j$. 
This completes the proof.
\end{proof}

\section{Approximation with a solution to the ODE system}

In this section, we prove the asymptotic formula 
(\ref{eq:asymptest}) in Theorem \ref{thm:main}. 
We introduce $(\alpha_1(t,\xi), \alpha_2(t,\xi))$ by
  \begin{equation}
    \alpha_j(\log t,\xi) = w_j(t,\xi)
  \end{equation}
  for $t > 0$, where $w_j(t,\xi)={{\mathcal F}}[U(-t)u_j(t,\cdot)](\xi)$. 
  We see that $(\alpha_1, \alpha_2)$ solves
  \begin{equation}
    i\partial_t\alpha_j = F_j(\alpha_1, \alpha_2) + \widetilde{r}_j
  \end{equation}
  for $j = 1, 2$, where
  \begin{equation}\label{eq:trj}
    \widetilde{r}_j(t, \xi) = e^tr_j(e^t, \xi).
  \end{equation}
We first give the error estimate in Corollary \ref{cor:rj}
in terms of $\tilde{r}_j$.
\begin{proposition}\label{prop:rj}
  Let $\tilde{r}_j$ be as in  \eqref{eq:trj}.
  Let $\delta$ and $\epsilon_0$ be as in Proposition \ref{prop:global}.
  If $\epsilon < \epsilon_0$, then it holds that
  \begin{equation}\label{eq:esttrj}
    \|\tilde{r}_j(t)\|_{L^2} \lesssim \epsilon^3e^{-(\frac12-3\delta\epsilon^2)t}, \quad
    \|\tilde{r}_j(t)\|_{L^\infty} \lesssim \epsilon^3e^{-(\frac14-3\delta\epsilon^2)t}
  \end{equation}
for $t > -1$.
\end{proposition}

\begin{proof}[Proof of Theorem \ref{thm:main}] 
We now prove \eqref{eq:asymptest} in Theorem  \ref{thm:main}. 
  We construct a one-parameter family 
  $\{(\phi_1(\cdot, \xi), \phi_2(\cdot, \xi))\}_{\xi \in \R}$ of solutions to
  the ODE system \eqref{eq:ODE} satisfying 
  \begin{equation}
    \notag
    \|\phi_j(t)-\alpha_j(t)\|_{L^2} \lesssim \epsilon^3
    e^{-(\frac12-3\delta\epsilon^2)t}, \quad
    \|\phi_j(t)-\alpha_j(t)\|_{L^\infty} \lesssim \epsilon^3e^{-(\frac14-3\delta\epsilon^2)t}
  \end{equation}
  for $t \ge 1$. 
Employing the idea of Hayashi-Li-Naumkin 
\cite{HLN}, we construct a sequence of approximate
  solution $\{(\phi_1^{(n)}, \phi_2^{(n)})\}_{n \in \N_0}$.
  We first define $\phi_j^{(0)}(t,\xi) := \alpha_j(t,\xi)$, namely $\phi_j^{(0)}$ satisfy
  \begin{equation}
    i\frac{d}{dt}\phi_j^{(0)} = F_j(\phi_1^{(0)}, \phi_2^{(0)}) + \tilde{r}_j.
  \end{equation}
  For $n \ge 1$, we inductively define $(\phi_1^{(n)}, \phi_2^{(n)})$ by
  \begin{equation}
    \left\{
    \begin{aligned}
      i\frac{d}{dt}\phi_1^{(n)}(t,\xi) & = F_1(\phi_1^{(n-1)}(t,\xi), \phi_2^{(n-1)}(t,\xi)), \quad t \ge 1, \\
      i\frac{d}{dt}\phi_2^{(n)}(t,\xi) & = F_2(\phi_1^{(n-1)}(t,\xi), \phi_2^{(n-1)}(t,\xi)), \quad t \ge 1,
    \end{aligned}
    \right.
  \end{equation}
  with the final state condition
  \begin{equation}
    \lim_{t \to \infty} \|\phi_j^{(n)}(t) - \alpha_j(t)\|_{L^\infty} = 0.
  \end{equation}
  Taking the difference between $\phi_j^{(0)}$ and $\phi_j^{(1)}$ and integrating on $[t,\infty)$, we have
  the integral form
  \begin{equation}\label{eq:if0}
    \phi_j^{(1)}(t,\xi) = \alpha_j(t,\xi) +i \int_t^\infty\tilde{r}_j(s, \xi) ds.
  \end{equation}
  Taking the difference between $\phi_j^{(n)}$ and $\phi_j^{(n-1)}$, we also have the integral form
  \begin{eqnarray}
   \lefteqn{\phi_j^{(n)}(t,\xi) -\phi_j^{(n-1)}(t,\xi)}\label{eq:ifn}\\
    &=&-i \int_t^\infty\left(F_j(\phi_1^{(n-1)}(s,\xi), \phi_2^{(n-1)}(s,\xi))
    - F_j(\phi_1^{(n-2)}(s,\xi), \phi_2^{(n-2)}(s,\xi))\right)ds
    \nonumber
  \end{eqnarray}
  for $n \ge 2$.
  Hereafter, we consider the integral form \eqref{eq:if0} and \eqref{eq:ifn}.
  We claim that $\{(\phi_1^{(n)}, \phi_2^{(n)})\}_{n \in \N_0}$ belongs to
   \begin{equation}
    Z := \left\{(a_1, a_2) \in (C([1,\infty)\times \R))^2\ ;\
    \sum_{j=1}^2 \|a_j(t) - \alpha_j(t)\|_{L^p} 
       \le C_1\epsilon^3e^{-(\frac14+\frac{1}{2p}-3\delta\epsilon^2)t} 
       \quad\text{for}\ p=2,\infty
    \right\}
  \end{equation}
  for suitable $C_1>0$. Since Proposition \ref{prop:global} implies that there exists $C_0$ such that
  \begin{equation}
    \sum_{j=1}^2\|\alpha_j(t)\|_{L^\infty([1,\infty)\times\R)} \le C_0\epsilon,
  \end{equation}
  we see that for any choice of $C_1 > 0$ there exists $\epsilon_0 > 0$ such that for any $\epsilon \in (0,\epsilon_0]$
  \begin{equation}
    \sum_{j=1}^2\|a_j(t)\|_{L^\infty([1,\infty)\times\R)} \le 2C_0\epsilon
  \end{equation}
  holds for any $(a_1, z_2) \in Z$.
  By the definition of the space $Z$, 
  we easily see $(\phi_1^{(0)}, \phi_2^{(0)}) \in Z$. 
  Further, thanks to \eqref{eq:esttrj}, there exists $C_2 > 0$ such that
  \begin{equation}
    \sum_{j=1}^2\|\phi_j^{(1)}(t) - \phi_j^{(0)}(t)\|_{L^p}
    = \sum_{j=1}^2\left\|\int_t^\infty\tilde{r}_j(s)ds\right\|_{L^p}
    \le C_2\epsilon^3e^{-(\frac14 + \frac{1}{2p}-3\delta\epsilon^2)t}
  \end{equation}
  for $p = 2, \infty$ and $t \ge 1$.
  We now choose $C_1 = 2C_2$. Then $(\phi^{(1)}_1, \phi^{(2)}_1) \in Z$.
  Now, suppose for some $k \ge 1$ that $(\phi^{(n)}_1, \phi^{(n)}_1) \in Z$ and
  \begin{equation}\label{eq:induction}
    \sum_{j=1}^2\|\phi^{(n)}_j(t) - \phi^{(n-1)}_j(t)\|_{L^p} \le 2^{-n+1}C_2\epsilon^3e^{-(\frac14+\frac{1}{2p}-3\delta\epsilon^2)t}
  \end{equation}
  hold for $p=2, \infty$, $t \ge 1$, and $n \in [1,k]$. Then, since
  \[
    |F_j(\phi^{(k)}_1, \phi^{(k)}_2) - F_j(\phi^{(k-1)}_1, \phi^{(k-1)}_2)|
    \le C\left(
    \sum_{j=1}^2|\phi^{(k)}_j| + \sum_{j=1}^2|\phi^{(k-1)}_j|
    \right)^2
    \sum_{j=1}^2|\phi^{(k)}_j - \phi^{(k-1)}_j|,
  \]
  we see from the assumption of the induction that $(\phi^{(k+1)}_j, \phi^{(k+1)}_j)$ is well-defined
  as an element of $C([0,\infty) \times \R)$ and the bound
  \begin{align*}
    \sum_{j=1}^2\|\phi^{(k+1)}_j(t) - \phi^{(k)}_j\|_{L^p}
     & \le 16C_0^2C\epsilon^2_0
     \int_t^\infty2^{-k+1}C_2\epsilon^3
     e^{-(\frac{1}{4}
    + \frac{1}{2p}-3\delta\epsilon^2)s}ds                                                           \\
     & \le (128C_0^2C\epsilon_0^2)2^{-k}
     C_2\epsilon^3e^{-(\frac14 + \frac{1}{2p}-3\delta\epsilon^2)t}
  \end{align*}
  holds for $p = 2, \infty$, and $t \ge 1$, which implies \eqref{eq:induction} for $n = k+1$
  if $\epsilon_0$ is chosen so that $128C_0^2C\epsilon_0^2 \le 1$. Moreover,
  \begin{equation} \nonumber
    \sum_{j=1}^2\|\phi_j^{(k+1)}(t) - \alpha_j(t)\|_{L^p}
    \le \sum_{\ell =1}^{k+1}\sum_{j=1}^2\|\phi_j^{(\ell)}(t) - \phi_j^{(\ell-1)}(t)\|_{L^p}
    \le 2C_2\epsilon^3e^{-(\frac{1}{4}+\frac{1}{2p}-3\delta\epsilon^2)t}
  \end{equation}
  holds for $p = 2, \infty$, and $t \ge 1$, and hence $(\phi_1^{(k+1)}, \phi_2^{(k+1)}) \in Z$.
  Thus an induction argument shows that $\{(\phi_1^{(n)}, \phi_2^{(n)})\}_{n \in \N_0} \in Z$
  is well-defined and \eqref{eq:induction} holds true for all $n \in \N$.
  One easily sees that $Z$ is complete with respect to the metric on $(L^\infty([1,\infty);L^2(\R)\cap L^\infty(\R)))^2$.
  Note that \eqref{eq:induction} implies that $\{(\phi_1^{(n)}, \phi_2^{(n)})\}_{n \in \N_0} \subset Z$ is
  a Cauchy sequence with respect to this metric. Hence, we finds the limit $(\phi_1, \phi_2)\in Z$.
  Let us confirm that the limit $(\phi_1, \phi_2)$ is a solution to \eqref{eq:ODE}.
  To this end, we note that \eqref{eq:induction} implies that $\{(\phi_1^{(n)}, \phi_2^{(n)})\}_{n \in \N_0}$
  converges to $(\phi_1, \phi_2)$ also in $(L^1(1,\infty;L^\infty(\R))^2)$.
  Hence,
  \begin{equation*}
    \int_t^\infty F_j(\phi_1^{(n)}, \phi_2^{(n)})(s,\xi) ds \to \int_t^\infty F_j(\phi_1, \phi_2)(s,\xi)ds
  \end{equation*}
  in $L^\infty([1,\infty)\times \R)$ as $n \to \infty$. Then, by using \eqref{eq:if0} and the definition of
  the sequence, it holds for each $t_1, t_2 \in [1, \infty)$ that
  \begin{align*}
    \phi_j(t_2) - \phi_j(t_1)
     & = \lim_{n \to \infty}(\phi_j^{(n)}(t_2) - \phi_j^{(n)}(t_1)) \\
     & = \phi_j^{(1)}(t_2) - \phi_j^{(1)}(t_1)
    - i\lim_{n \to \infty}\int_{t_1}^{t_2} F_j(\phi_1^{(n-1)}, \phi_2^{(n-1)})ds
    + i\int_{t_1}^{t_2} F_j(\phi_1^{(0)}, \phi_2^{(0)})ds           \\
     & = - i \int_{t_1}^{t_2} F_j(\phi_1, \phi_2)ds.
  \end{align*}
  This shows that $(\phi_1, \phi_2)$ is a solution to \eqref{eq:ODE}.

  Finally, we show the asymptotic estimate \eqref{eq:asymptest}.
  We see from Proposition \ref{prop:rj} and the fact that $(\phi_1, \phi_2) \in Z$ that
  \begin{align*}
    u_j(t) & = (it)^{-\frac12}e^{\frac{i|x|^2}{2t}}w_j(t, \tfrac{x}{t})
    + O(\epsilon t^{-\frac34 +\frac{1}{2p}+\delta\epsilon^2})                              \\
           & = (it)^{-\frac12}e^{\frac{i|x|^2}{2t}}\alpha_j(\log t, \tfrac{x}{t})
    + O(\epsilon t^{-\frac34+\frac{1}{2p}+\delta\epsilon^2})                               \\
           & = (it)^{-\frac12}e^{\frac{i|x|^2}{2t}}\phi_j(\log t, \tfrac{x}{t})
    + O(\epsilon^3t^{-\frac34+\frac{1}{2p}+3\delta\epsilon^2})
    + O(\epsilon t^{-\frac34+\frac{1}{2p}+\delta\epsilon^2})                               \\
           & = (it)^{-\frac12}e^{\frac{i|x|^2}{2t}}\phi_j(\log t, \tfrac{x}{t})
    + O(\epsilon t^{-\frac34+\frac{1}{2p}+\delta\epsilon^2})
  \end{align*}
  in $L^p(\R)$ as $t \to \infty$ for $p = 2, \infty$. The non-endpoint case $p \in (2,\infty)$
  follows by interpolation. Writing $3\delta$ as $\delta$, we obtain \eqref{eq:asymptest}. 
  This completes the proof of Theorem \ref{thm:main}. 
\end{proof}

\section{Normalization of \eqref{eq:2NLS}}

In this section, we prove Theorems \ref{thm:standard1} 
and \ref{thm:standard2}. To this end, we note that 
the system \eqref{eq:2NLS} and \eqref{eq:ODE} are closed under the linear transform of the unknowns.
The change of the system is described
as a matrix manipulation.
\begin{proposition}[\cite{MaSeUr2}]\label{P:change}
Let $(\mathscr{A},\mathscr{V})$ be a matrix-vector representation of a system of the form \eqref{eq:2NLS} or \eqref{eq:ODE}
with unknown $(u_1,u_2)$.
Define a new unknown $(v_1,v_2)$ via
\[
	\begin{bmatrix} v_1 \\ v_2 \end{bmatrix} = \mathscr{M} \begin{bmatrix} u_1 \\ u_2 \end{bmatrix},\quad
	\mathscr{M} = \begin{bmatrix} a & b \\ c & d \end{bmatrix} \in \text{GL}_2(\R).
\]
Let $(\mathscr{A}',\mathscr{V}')$ be the matrix-vector representation of the system for $(v_1,v_2)$. Then,
\begin{equation}\label{E:Aprime}
	\mathscr{A}' = \frac{1}{\det \mathscr{M}}  \mathscr{D}(\mathscr{M}) \mathscr{A} \mathscr{D}(\mathscr{M})^{-1}
\end{equation}
and
\begin{equation}\label{E:Vprime}
	\mathscr{V}'
	= \frac1{\det \mathscr{M}}  \mathscr{D}(\mathscr{M}) \mathscr{V}
\end{equation}
hold, where
\[
	\mathscr{D}(\mathscr{M}) := \frac1{\det \mathscr{M}}
	\begin{bmatrix} d^2 & -2cd & c^2 \\ -bd & ad+ bc & -ac \\ b^2 & -2ab & a^2 \end{bmatrix}
	\in {SL}_3 (\R),
\]
\begin{equation*}
	\mathscr{D}(\mathscr{M})^{-1} =\frac1{\det \mathscr{M}}\begin{bmatrix} a^2 & 2ac & c^2 \\ ab & ad+ bc & cd \\ b^2 & 2bd & d^2 \end{bmatrix} \in SL_3(\R).
\end{equation*}
Moreover, 
$\vec{p}\in \R^3$ is an eigenvector of $\mathscr{A}$ associated with an eigenvalue $\lambda \in \C$ if and only if
$\mathscr{D}(\mathscr{M})\vec{p}\in \R^3$ is an eigenvector of $\mathscr{A}'$ associated with an eigenvalue $(\det \mathscr{M})^{-1}\lambda \in \C$.
Furthermore, for any $\mathscr{M}\in GL_2(\R)$, $\mathscr{D}(\mathscr{M})\vec{p}\in \mathcal{P}_+$ if and only if $\vec{p}\in \mathcal{P}_+$.
\end{proposition}

\begin{proof}[Proof of Theorem \ref{thm:standard1}]
We consider ODE system \eqref{eq:ODE} instead of \eqref{eq:2NLS}.
It is sufficient because the change of parameters of the nonlinearity does not depend on the linear part of the equation.

Pick an ODE system \eqref{eq:ODE} satisfying Assumption \ref{A1} and let
$\mathscr{A}^{(0)}$ be the matrix part of the system.
Let $(u_1^{(0)},u_2^{(0)})$ be the unknown for the system. 
Let $\lambda_1^{(0)}$ and $\lambda_2^{(0)}$ be two eigenvalues such that
$W(\lambda_j^{(0)};\mathscr{A}^{(0)}) \cap \mathcal{P}_+ \neq \emptyset$.
We apply several changes of variables to make the system is of the desired standard form.

{\bf Case 1}.
We first consider the case where $\mathscr{A}^{(0)}$ has three distinct eigenvalues.
Let $\lambda_0^{(0)}$ be the third eigenvalue.
By changing $(u_1^{(0)},u_2^{(0)})\mapsto (u_1^{(0)},-u_2^{(0)})$ if necessary and
 by relabeling the eigenvalues if necessary, we may suppose  without loss of generality that either
\begin{itemize}
\item[(a)] $\lambda_0^{(0)}>0$ and $\lambda_1^{(0)}>\lambda_2^{(0)}$;
\item[(b)] $\lambda_0^{(0)}=0$, $\lambda_2^{(0)}<0$, and $\lambda_2^{(0)}<\lambda_1^{(0)}\le |\lambda_2^{(0)}| $.
\end{itemize}
(Note that the former manipulation changes the sign of all the eigenvalues.)

{\bf Subcase 1-(a)}.
We first consider the subcase (a).
Pick
\[
	\vec{p} = {}^t\!( p_{1} ,
	p_{2} , p_{3})
\in W(\lambda_1^{(0)},\mathscr{A}^{(0)}) \cap \mathcal{P}_+
\]
with $p_{1}>0$ and
 introduce a new variable $(u_1^{(1)},u_2^{(1)})$ by
\[
	\begin{bmatrix}
	u_1^{(1)}\\
	u_2^{(1)}
	\end{bmatrix}
	=
	\mathscr{M}^{(1)}
	\begin{bmatrix}
	u_1^{(0)}\\
	u_2^{(0)}
	\end{bmatrix}, \quad
	\mathscr{M}^{(1)}=\begin{bmatrix}
	p_{1}^{1/2} & p_{1}^{-1/2}p_{2} \\
	0 & (\tfrac{p_{1}p_{3}-p_{2}^2}{p_{1}})^{1/2}
	\end{bmatrix}.
\]
Let $\mathscr{A}^{(1)}$ be the matrix corresponding to the system 
for $(u_1^{(1)},u_2^{(1)})$. 
A computation shows 
$\mathscr{D}(\mathscr{M}^{(1)}) \vec{p}=
\sqrt{p_1p_3-p_2^2}\ {}^t\!(1,0,1)$, which is a rephrase of the square completion
\[
	p_{1}|u_1^{(0)}|^2 + 2 p_{2} \Re (\overline{u_1^{(0)}}u_2^{(0)}) + p_{3} |u_2^{(0)}|^2 =
	|u_1^{(1)}|^2 + |u_2^{(1)} |^2.
\]
One sees from Proposition \ref{P:change} that 
${}^t\!(1,0,1) \in W(\lambda_1^{(1)};\mathscr{A}^{(1)})$,
where
$\lambda_1^{(1)}:=(\det \mathscr{M}^{(1)})^{-1} \lambda_1^{(0)}$.
Note that $\lambda_2^{(1)}:=(\det \mathscr{M}^{(1)})^{-1} \lambda_2^{(0)}$ satisfies $W(\lambda_2^{(1)};\mathscr{A}^{(1)}) \cap \mathcal{P}_+ \neq \emptyset$.
The third eigenvalue is $\lambda_0^{(1)}:=(\det \mathscr{M}^{(1)})^{-1} \lambda_0^{(0)}>0$.

We next adjust the magnitude of the matrix so that the third eigenvalue becomes $1$.
Define $(u_1^{(2)},u_2^{(2)})$ by
$u_j^{(2)}=(\lambda_0^{(1)})^{\frac12}u_j^{(1)}$.
Let $\mathscr{A}^{(2)}$ be the matrix corresponding to the system for $(u_1^{(2)},u_2^{(2)})$. 
We remark that the positive and negative eigenvalues are 
changed into $\lambda_1:=\lambda_1^{(1)}(\lambda_0^{(1)})^{-1}\not\in \{0,1\}$ and $\lambda_2:=\lambda_2^{(1)}(\lambda_0^{(1)})^{-1}\not\in \{0,1\}$, respectively.
Then, one has $\lambda_1>\lambda_2$ and 
\[
	{}^t\!(1,0,1) \in W(\lambda_1;\mathscr{A}^{(2)}), \quad
	W(\lambda_2;\mathscr{A}^{(2)}) \cap \mathcal{P}_+ \neq \emptyset, \quad
	W(1;\mathscr{A}^{(2)}) \neq \{0\}
\]
as desired.

We then apply further change of variable to simplify the eigenspace associated with eigenvalue $\lambda_2$.
Pick
\[
	\vec{q} = {}^t\!( q_{1} ,
	q_{2}, q_{3})
\in W(\lambda_2;\mathscr{A}^{(2)}) \cap \mathcal{P}_+.
\]
Define $(u_1^{(3)},u_2^{(3)})$ by
\[
	\begin{bmatrix}
	u_1^{(3)}\\
	u_2^{(3)}
	\end{bmatrix}
	=
	\mathscr{M}^{(3)}
	\begin{bmatrix}
	u_1^{(2)}\\
	u_2^{(2)}
	\end{bmatrix}, \quad
	\mathscr{M}^{(3)}=\begin{bmatrix}
	\cos \theta & -\sin \theta \\
	\sin \theta & \cos \theta
	\end{bmatrix}
\]
for $\theta \in \R$ and 
let $\mathscr{A}^{(3)}$ be the matrix corresponding to the system for $(u_1^{(3)},u_2^{(3)})$. 
This does not change eigenvalues since $\det \mathscr{M}^{(3)}=1$.
One has $\mathscr{D}(\mathscr{M}^{(3)}){}^t\!(1,0,1)={}^t\!(1,0,1)$
 for any $\theta$, which implies
${}^t\!(1,0,1) \in W(\lambda_1;\mathscr{A}^{(3)})$ for any $\theta$. 
We choose $\theta$ so that
$\mathscr{D}(\mathscr{M}^{(3)})\vec{q}$ is of the form ${}^t\!(\alpha,0,\beta)$.
If $q_2=0$, we let $\theta=0$.
Otherwise, since the second component of the vector $\mathscr{D}(\mathscr{M}^{(3)})\vec{q}$ is explicitly given as
\[
	\tfrac{q_3-q_1}{2} \sin 2\theta + q_2 \cos 2 \theta ,
\]
we choose $\theta$ so that $\cot 2\theta = \frac{q_1-q_3}{2q_2}$.
Then, one sees that 
\[
	W(\lambda_2 ;\mathscr{A}^{(3)}) =\Span
	\{ {}^t\!(1,0,k) \}
\]
for some positive $k \neq 1$.
Note that $k$ is positive since
$W(\lambda_2;\mathscr{A}^{(3)}) \cap \mathcal{P}_+ \neq \emptyset$ and that $k\neq1$ follows from $W(\lambda_1;\mathscr{A}^{(3)}) \cap W(\lambda_2;\mathscr{A}^{(3)}) = \{0\}$.

If $k>1$ then we further apply
\[
	\begin{bmatrix}
	u_1^{(4)}\\
	u_2^{(4)}
	\end{bmatrix}
	=
	\mathscr{M}^{(4)}
	\begin{bmatrix}
	u_1^{(3)}\\
	u_2^{(3)}
	\end{bmatrix}, \quad
	\mathscr{M}^{(4)}=\begin{bmatrix}
	0 & 1 \\
	-1 & 0
	\end{bmatrix}.
\]
Let $\mathscr{A}^{(4)}$ be the matrix corresponding the system for $(u_1^{(4)},u_2^{(4)})$. 
Since $\det \mathscr{M}^{(4)}=1$, the eigenvalues do not change. 
Further, since
\[
	\mathcal{D}(\mathscr{M}^{(4)})
	=\begin{bmatrix}
	0 & 0 & 1 \\
	0 & -1 & 0 \\
	1 & 0 & 0
	\end{bmatrix},
\]
we see from Proposition \ref{P:change}
that $W(\lambda_1;\mathscr{A}^{(4)}) = \Span \{{}^t\!(1,0,1) \}$
and
\[
	W(\lambda_2;\mathscr{A}^{(4)}) = \Span \{{}^t\!(k,0,1) \}
	= \Span \{{}^t\!(1,0,k^{-1}) \}.
\]
If $0<k<1$ then we 
simply let $\mathscr{A}^{(4)}=\mathscr{A}^{(3)}$.
In either case, one has
\[
	W(1; \mathscr{A}^{(4)})
	= \Span \{{}^t\! (\eta_2,1,\eta_3)\}
\]
with $\eta_2,\eta_3\in\R$. Note that the second component of the eigenvector does not vanish since otherwise we have $W(\lambda_1;\mathscr{A}^{(4)})\oplus W(\lambda_2;\mathscr{A}^{(4)})\oplus W(1;\mathscr{A}^{(4)})  \subsetneq \R^3$, a contradiction.

Thus, by slightly changing the notation,
we have
\begin{align*}
	W(\lambda_1;\mathscr{A}^{(4)}) 
	&= \Span \{{}^t\!(1,0,1) \}, \\
	W(\lambda_2;\mathscr{A}^{(4)}) 
	&= \Span \{{}^t\!(1+\eta_1,0,\eta_1) \}, \\
	W(1;\mathscr{A}^{(4)}) 
	&= \Span \{{}^t\!(\eta_2,1,\eta_3) \}
\end{align*}
with $\lambda_1,\lambda_2 \in \R\setminus\{0,1\}
$, $\lambda_1>\lambda_2$,
$\eta_1>0$, and $\eta_2,\eta_3\in \R$.
Hence,
\begin{align*}
	\mathscr{A}^{(4)}
	={}& \begin{bmatrix}
	1 & 1+\eta_1 & \eta_2 \\
	0 & 0& 1\\
	1 & \eta_1& \eta_3
	\end{bmatrix}
	\begin{bmatrix}
	\lambda_1 & 0 & 0 \\
	0 & \lambda_2 & 0\\
	0 & 0 & 1
	\end{bmatrix}
	\begin{bmatrix}
	1 & 1+\eta_1 & \eta_2 \\
	0 & 0& 1\\
	1 & \eta_1& \eta_3
	\end{bmatrix}^{-1} \\
	={}& \begin{bmatrix}
	\lambda_1 - (1+\eta_1)(\lambda_1-\lambda_2) & 
	\eta_2(1-\lambda_1) +(1+\eta_1)(\eta_2-\eta_3)(\lambda_1-\lambda_2) 
	& (1+\eta_1)(\lambda_1-\lambda_2) \\
	0 & 1& 0\\
	-\eta_1(\lambda_1-\lambda_2) & \eta_3(1- \lambda_1) +\eta_1(\eta_2-\eta_3)(\lambda_1-\lambda_2)& \lambda_1 +\eta_1(\lambda_ 1- \lambda_2)
	\end{bmatrix},
\end{align*}
which is the desired standard form.

{\bf Subcase 1-(b)}.
We then move onto the subcase (b).
Define $\mathscr{A}^{(1)}$ as in the subcase (a).
We modified the definition of $\mathscr{A}^{(2)}$. Set the magnitude of the matrix so that $\lambda_2$ becomes $-1$.
Define $(u_1^{(2)},u_2^{(2)})$ by
$u_j^{(2)}=|\lambda_2^{(1)}|^{\frac12}u_j^{(1)}$.
Let $\mathscr{A}^{(2)}$ be the matrix corresponding to the system for $(u_1^{(2)},u_2^{(2)})$. 
Set $\lambda_1:= \lambda_1^{(1)}|\lambda_2^{(1)}|^{-1} \in (-1,0)\cup(0,1]$.
One has
\[
	{}^t\!(1,0,1) \in W(\lambda_1;\mathscr{A}^{(2)}), \quad
	W(-1;\mathscr{A}^{(2)}) \cap \mathcal{P}_+ \neq \emptyset, \quad
	\rank \mathscr{A}^{(2)}=2.
\]
The rest is the same as in subcase (a).
By defining $\mathscr{A}^{(3)}$ and $\mathscr{A}^{(4)}$ via exactly the same way,
one obtains
\begin{align*}
	W(\lambda_1;\mathscr{A}^{(4)}) 
	&= \Span \{{}^t\!(1,0,1) \}, \\
	W(-1;\mathscr{A}^{(4)}) 
	&= \Span \{{}^t\!(1+\eta_1,0,\eta_1) \}, \\
	W(0;\mathscr{A}^{(4)}) 
	&= \Span \{{}^t\!(\eta_2,1,\eta_3) \}
\end{align*}
with $\lambda_1\in (-1,0)\cup(0,1]$,
$\eta_1>0$, and $\eta_2,\eta_3\in \R$.
Hence,
\begin{align*}
	\mathscr{A}^{(4)}
	={}& \begin{bmatrix}
	1 & 1+\eta_1 & \eta_2 \\
	0 & 0& 1\\
	1 & \eta_1& \eta_3
	\end{bmatrix}
	\begin{bmatrix}
	\lambda_1 & 0 & 0 \\
	0 & -1 & 0\\
	0 & 0 & 0
	\end{bmatrix}
	\begin{bmatrix}
	1 & 1+\eta_1 & \eta_2 \\
	0 & 0& 1\\
	1 & \eta_1& \eta_3
	\end{bmatrix}^{-1} \\
	={}& \begin{bmatrix}
	\lambda_1 - (1+\eta_1)(\lambda_1+1) & 
	-\eta_2\lambda_1 +(1+\eta_1)(\eta_2-\eta_3)(\lambda_1+1) 
	& (1+\eta_1)(\lambda_1+1) \\
	0 & 0 & 0\\
	-\eta_1(\lambda_1+1) & -\eta_3 \lambda_1 +\eta_1(\eta_2-\eta_3)(\lambda_1+1)& \lambda_1 +\eta_1(\lambda_ 1+1)
	\end{bmatrix},
\end{align*}
which is of the desired form.

{\bf Case 2}.
Let us consider the exceptional case, i.e.,
$\lambda_0^{(0)}$ is equal to
$\lambda_1^{(0)}$ or $\lambda_2^{(0)}$.
By changing $(u_1^{(0)},u_2^{(0)})\mapsto (u_1^{(0)},-u_2^{(0)})$ if necessary and
 by relabeling the eigenvalues if necessary, we may suppose  without loss of generality that 
$\mathscr{A}^{(0)}$ has exactly two distinct eigenvalues $\lambda_0^{(0)}>0$ and $\lambda^{(0)}\not \in \{ 0, \lambda_0^{(0)}\}$. By assumption, we have
$W(\lambda_0^{(0)}) \cap \mathcal{P}_+\neq \emptyset$ and $W(\lambda^{(0)}) \cap \mathcal{P}_+\neq \emptyset$.
There exist two subcases;
\begin{itemize}
\item[(a)] $\dim W(\lambda_0^{(0)}) =2$;
\item[(b)] $\dim W(\lambda_0^{(0)}) =1$, \end{itemize}
%
In the both cases, the procedure in Subcase 1-(a) works.
The only difference is that one of $\lambda_1$ and $\lambda_2$ is equal to $\lambda_0$.
Further, in the Subcase 2-(b),
$\mathscr{A}^{(0)}$ is not diagonalized and
the Jordan block of size two appears.
One has
\begin{align*}
	W(1;\mathscr{A}^{(4)}) 
	&\supset \Span \{{}^t\!(1,0,1) \}, \\
	\tilde{W}(1;\mathscr{A}^{(4)}) 
	&= \Span \{{}^t\!(1,0,1), {}^t\!(\eta_2,1,\eta_3) \}, \\
	W(\lambda;\mathscr{A}^{(4)}) 
	&= \Span \{{}^t\!(1+\eta_1,0,\eta_1) \}
\end{align*}
with $\lambda \in \R\setminus\{0,1\}
$, $\eta_1>0$, and $\eta_2,\eta_3\in \R$.
Hence, 
\begin{align*}
	\mathscr{A}^{(4)}
	&= \begin{bmatrix}
	1 & 1+\eta_1 & \eta_2 \\
	0 & 0& 1\\
	1 & \eta_1& \eta_3
	\end{bmatrix}
	\begin{bmatrix}
	1 & 0 & \eta_4 \\
	0 & \lambda & 0\\
	0 & 0 & 1
	\end{bmatrix}
	\begin{bmatrix}
	1 & 1+\eta_1 & \eta_2 \\
	0 & 0& 1\\
	1 & \eta_1& \eta_3
	\end{bmatrix}^{-1}\\
	&=\begin{bmatrix}
	1 - (1+\eta_1)(1-\lambda) & 
	 (1+ \eta_1)(\eta_2 - \eta_3)(1-\lambda) + \eta_4
	& (1+\eta_1)(1-\lambda) \\
	0 & 1 & 0\\
	-\eta_1(1-\lambda) & \eta_1(\eta_2- \eta_3)(1-\lambda)+\eta_4 & 1+\eta_1(1-\lambda)
	\end{bmatrix}
\end{align*}
with $\lambda \not \in\{0,1\}$, 
$\eta_1>0$, and $\eta_2,\eta_3,\eta_4\in \R$.
Note that $\eta_4=0$ and $\eta_4\neq0$ correspond to the Subcases 2-(a) and 2-(b), respectively.
This is the desired standard forms.
%
\end{proof}

\begin{proof}[Proof of Theorem \ref{thm:standard2}]
Pick an ODE system \eqref{eq:ODE} satisfying Assumption \ref{A2} and let
$\mathscr{A}^{(0)}$ be the matrix part of the system.
Let $(u_1^{(0)},u_2^{(0)})$ be the unknown for the system. 
Let $\lambda_j^{(0)}$ be the three distinct eigenvalues such that
${\lambda_1^{(0)}}/{\lambda_3^{(0)}}<
{\lambda_2^{(0)}}/{\lambda_3^{(0)}}<1$,
$(W(\lambda_1^{(0)};\mathscr{A}^{(0)}) \oplus W(\lambda_2^{(0)};\mathscr{A}^{(0)})) \cap \mathcal{P}_+ \neq \emptyset$,
$W(\lambda_3^{(0)};\mathscr{A}^{(0)} )\cap \mathcal{P}_+ \neq \emptyset$,
and $W(\lambda_1^{(0)};\mathscr{A}^{(0)})\cap \mathcal{P}_+=W(\lambda_2^{(0)};\mathscr{A}^{(0)})\cap \mathcal{P}_+=\emptyset$.
By changing $(u_1^{(0)},u_2^{(0)})\mapsto (u_1^{(0)},-u_2^{(0)})$ if necessary, we may suppose that $\lambda_3^{(0)}<0$ without loss of generality.
One has ${\lambda_1^{(0)}}>
{\lambda_2^{(0)}}>\lambda_3^{(0)}$.
There are four subcases:
\begin{itemize}
\item[(a)] $W(\lambda_1^{(0)};\mathscr{A}^{(0)}) \subset  \mathcal{P}_0$ and $W(\lambda_2^{(0)};\mathscr{A}^{(0)}) \subset  \mathcal{P}_0$;
\item[(b)] $W(\lambda_1^{(0)};\mathscr{A}^{(0)}) \subset \mathcal{P}_0$ and $W(\lambda_2^{(0)};\mathscr{A}^{(0)}) \cap \mathcal{P}_0 = \{0\}$;
\item[(c)] $W(\lambda_1^{(0)};\mathscr{A}^{(0)}) \cap \mathcal{P}_0 = \{0\}$ and
$W(\lambda_2^{(0)};\mathscr{A}^{(0)}) \subset \mathcal{P}_0$;
\item[(d)]
$W(\lambda_1^{(0)};\mathscr{A}^{(0)}) \cap \mathcal{P}_0 = \{0\}$ and $W(\lambda_2^{(0)};\mathscr{A}^{(0)}) \cap \mathcal{P}_0 = \{0\}$.
\end{itemize}

{\bf Subcase (a)}. 
By assumption, one can choose $\theta_1 ,\theta_2 \in \R/2\pi \Z$ with $\theta_1\neq \theta_2$ such that
\[
	\vec{p}_j = 
	{}^t\!(
	 1+ \sin \theta_j,
	 \cos \theta_j,
	  1- \sin \theta_j )\in W(\lambda_j^{(0)};\mathscr{A}^{(0)}) \]
for $j=1,2$. With this $\theta_j$ and $\sigma_0\in \{ \pm1\}$,
we introduce
\[
	\begin{bmatrix}
	u_1^{(1)}\\
	u_2^{(1)}
	\end{bmatrix}
	=
	\mathscr{M}^{(1)}
	\begin{bmatrix}
	u_1^{(0)}\\
	u_2^{(0)}
	\end{bmatrix}, \quad
	\mathscr{M}^{(1)}=\begin{bmatrix}
	\sigma_0\sqrt{1+\sin \theta_1} & \sigma_0 \frac{\cos \theta_1}{\sqrt{1+\sin \theta_1}} \\
	\sqrt{1+\sin \theta_2} &  \frac{\cos \theta_2}{\sqrt{1+\sin \theta_2}}
	\end{bmatrix}
\]
if $\theta_1,\theta_2\neq3\pi/2$.
If $\theta_1=3\pi/2 \neq \theta_2$
or $\theta_1 \neq 3\pi/2 = \theta_2$, we define 
\[
	\begin{bmatrix}
0 & -\sqrt2 \\
	\sqrt{1+\sin \theta_2} &  \frac{\cos \theta_2}{\sqrt{1+\sin \theta_2}}
	\end{bmatrix}\quad \text{or}\quad
	\begin{bmatrix}
	\sqrt{1+\sin \theta_1} & \frac{\cos \theta_1}{\sqrt{1+\sin \theta_1}} \\
	0 & \sqrt{2}
	\end{bmatrix},
\]
respectively. We choose $\sigma_0 \in \{\pm1\}$ so that $\det \mathscr{M}^{(1)}>0$.
Let $\mathscr{A}^{(1)}$ be the matrix corresponds to the system for $(u_1^{(1)},u_2^{(1)})$. One sees by a direct computation that
\[
	\mathscr{D}(\mathscr{M}^{(1)}) \vec{p}_1 = {}^t\!(c,0,0), \quad
	\mathscr{D}(\mathscr{M}^{(1)}) \vec{p}_2 = {}^t\!(0,0,c)
\]
with the constant $c=(\det\mathscr{M}^{(1)})^{-1} 2(1-\cos (\theta_1-\theta_2)) \neq0$.
Hence,
\[
	W(\lambda_1^{(1)};\mathscr{A}^{(1)}) = \Span\{{}^t\!(1,0,0)\}, \quad
	W(\lambda_2^{(1)};\mathscr{A}^{(1)}) = \Span\{{}^t\!(0,0,1)\},
\]
where $\lambda_j^{(1)}:=(\det \mathscr{M}^{(1)})^{-1} \lambda_j^{(0)}$.
One has
\[
	W(\lambda_3^{(1)};\mathscr{A}^{(1)}) = \Span\{{}^t\!(\eta_2,1,\eta_3)\}
\]
for some $\eta_2,\eta_3\in\R$ with
$\eta_2\eta_3>1$. The condition comes from $(\eta_2,1,\eta_3) \in \mathcal{P}_+$.

We next apply
\[
	\begin{bmatrix}
	u_1^{(2)}\\
	u_2^{(2)}
	\end{bmatrix}
	=
	\mathscr{M}^{(2)}
	\begin{bmatrix}
	u_1^{(1)}\\
	u_2^{(1)}
	\end{bmatrix}, \quad
	\mathscr{M}^{(2)}=\begin{bmatrix}
	(\eta_2/\eta_3)^{1/4} |\lambda_3^{(1)}|^{1/2} & 0 \\
	0 & (\eta_2/\eta_3)^{-1/4} |\lambda_3^{(1)}|^{1/2}
	\end{bmatrix}.
\]
Let $\mathscr{A}^{(2)}$ be the matrix corresponds to the system for $(u_1^{(2)},u_2^{(2)})$. 
Let $\lambda_j:= \lambda_j^{(1)}
|\lambda_3^{(1)}|^{-1}$ for $j=1,2$.
One has $\lambda_1>\lambda_2>\lambda_3=-1$ and
\[
	W(\lambda_1;\mathscr{A}^{(2)}) = \Span\{{}^t\!(1,0,0)\}, \quad
	W(\lambda_2;\mathscr{A}^{(2)}) = \Span\{{}^t\!(0,0,1)\}, \quad
	W(-1;\mathscr{A}^{(2)}) = \Span\{{}^t\!(\eta,1,\eta)\}
\]
with a number $\eta\in (-\infty,-1)\cup (1,\infty)$. 
Thus,
\begin{align*}
	\mathscr{A}^{(2)}
	={}& \begin{bmatrix}
	1 & 0 & \eta \\
	0 & 0& 1\\
	0 & 1& \eta
	\end{bmatrix}
	\begin{bmatrix}
	\lambda_1 & 0 & 0 \\
	0 & \lambda_2 & 0\\
	0 & 0 & -1
	\end{bmatrix}
	\begin{bmatrix}
	1 & 0 & \eta \\
	0 & 0& 1\\
	0 & 1& \eta
	\end{bmatrix}^{-1} 
	= \begin{bmatrix}
	\lambda_1 & 
	-\eta(\lambda_1+1) 
	& 0 \\
	0 & -1& 0\\
	0 & -\eta(\lambda_2+1)  & \lambda_2
	\end{bmatrix}.
\end{align*}
This is of the desired form.

{\bf Subcases (b) and (c)}. 
Let us consider Subcase (b). 
We generalize the assumption $\lambda_1^{(0)}>\lambda_2^{(0)}$ into $\lambda_1^{(0)}\neq \lambda_2^{(0)}$ in this subcase
so that Subcase (c) is handled in a unified way.
As in Subcase (a), one can find $\theta_2 \in \R /2\pi \Z$ such that
\[
	\vec{p}_2 = {}^t\!(
	 1+ \sin \theta_2,
	  \cos \theta_2, 
	  1- \sin \theta_2 )\in W(\lambda_2^{(0)};\mathscr{A}^{(0)}) \]
With this $\theta_2$, we introduce
\[
	\begin{bmatrix}
	u_1^{(1)}\\
	u_2^{(1)}
	\end{bmatrix}
	=
	\mathscr{M}^{(1)}
	\begin{bmatrix}
	u_1^{(0)}\\
	u_2^{(0)}
	\end{bmatrix}, \quad
	\mathscr{M}^{(1)}=\begin{bmatrix}
	\sqrt{1+\sin \theta_2} &  \frac{\cos \theta_2}{\sqrt{1+\sin \theta_2}} \\
	0 & 1/\sqrt{1+\sin\theta_2}
	\end{bmatrix}\in SL_2(\R)
\]
if $\theta_2\neq3\pi/2$.
If $\theta_2=3\pi/2$, we define 
\[
	\mathscr{M}^{(1)}=\begin{bmatrix}
0 & \sqrt2 \\
	-1/\sqrt2 & 0
	\end{bmatrix}\in SL_2(\R).
\]
Let $\mathcal{A}^{(1)}$ be the new matrix part.
One sees that the eigenvalues do not change and
\[
	W(\lambda_2^{(0)};\mathcal{A}^{(1)})
	= \Span \{ {}^t\!(1,0,0) \}.
\]
Note that
\[
	W(\lambda_1^{(0)};\mathcal{A}^{(1)})
	= \Span \{ {}^t\!(a,b,1) \}
\]
with some numbers $a,b \in \R$ such that $a<b^2$.
Note that the third component of a nonzero element of $W(\lambda_1^{(0)};\mathcal{A}^{(1)})$ do not vanish since otherwise 
$(W(\lambda_1^{(0)};\mathcal{A}^{(1)}) \oplus W(\lambda_2^{(0)};\mathcal{A}^{(1)})) \cap \mathcal{P}_+ = \emptyset$, a contradiction.

We next apply the following change of variables:
\[
	\begin{bmatrix}
	u_1^{(2)}\\
	u_2^{(2)}
	\end{bmatrix}
	=
	\mathscr{M}^{(2)}
	\begin{bmatrix}
	u_1^{(1)}\\
	u_2^{(1)}
	\end{bmatrix}, \quad
	\mathscr{M}^{(2)}=|\lambda_3|^{\frac12}(1+b)^{-\frac12}\begin{bmatrix}
	1& 0  \\
	\sqrt{b^2-a} & 1+b
	\end{bmatrix}.
\]
Let $\mathscr{A}^{(2)}$ be the new matrix.
Then, since $\det \mathscr{M}^{(2)}=|\lambda_3|>0$, one has 
\[
	W(\lambda_1;\mathscr{A}^{(2)})
	= \Span\{{}^t\!(1,0,-1)\}, \quad
	W(\lambda_2;\mathscr{A}^{(2)})
	= \Span\{{}^t\!(1,0,0)\}, \quad
	W(-1;\mathscr{A}^{(2)})\cap \mathcal{P}_+ \neq \emptyset,
\]
where $\lambda_j:= \lambda_j^{(0)}
|\lambda_3^{(0)}|^{-1}$ for $j=1,2$.
Then, $\lambda_1>-1$ and $\lambda_2>-1$.
Thus, we obtain
\begin{align*}
	\mathscr{A}^{(2)}
	={}& \begin{bmatrix}
	1 & 1 & \eta_2 \\
	0 & 0& 1\\
	-1 & 0& \eta_3
	\end{bmatrix}
	\begin{bmatrix}
	\lambda_1 & 0 & 0 \\
	0 & \lambda_2 & 0\\
	0 & 0 & -1
	\end{bmatrix}
	\begin{bmatrix}
	1 & 1 & \eta_2 \\
	0 & 0& 1\\
	-1 & 0& \eta_3
	\end{bmatrix}^{-1} \\
	={}& \begin{bmatrix}
	\lambda_2 & 
	-\eta_2(\lambda_2+1)  + \eta_3 (\lambda_1-\lambda_2)
	& -(\lambda_1-\lambda_2) \\
	0 & -1& 0\\
	0 & -\eta_3(\lambda_1+1)  & \lambda_1
	\end{bmatrix}
\end{align*}
with $\eta_2\eta_3>1$.
This is the special case $\eta_1=0$ of $\mathscr{A}_{2,2}$.

{\bf Subcase (d)}. 
One writes
\[
	W(\lambda_1^{(0)};\mathscr{A}^{(0)}) = \Span \{ {}^t\!(p_1,p_2,p_3)\}
\]
with $p_1p_3<p_2^2$.
We introduce
\[
	\begin{bmatrix}
	u_1^{(1)}\\
	u_2^{(1)}
	\end{bmatrix}
	=
	\mathscr{M}^{(1)}
	\begin{bmatrix}
	u_1^{(0)}\\
	u_2^{(0)}
	\end{bmatrix}
\]
with
\[
	\mathscr{M}^{(1)}=\begin{bmatrix}
	\sigma_0 p_1 & \sigma_0 p_2  \\
	0 & (p_2^2-p_1p_3)^{1/2}
	\end{bmatrix}
\]
if $p_1\neq 0$,
\[
	\mathscr{M}^{(1)}=\begin{bmatrix}
	\sigma_0 p_2 & 0 \\
	p_2  & p_3
	\end{bmatrix}
\]
if $p_1=0$ and $p_3\neq 0$, and
\[
	\mathscr{M}^{(1)}=\begin{bmatrix}
	1 & 1 \\
	-1 & 1
	\end{bmatrix}
\]
if $p_1= p_3 = 0$, where $\sigma_0\in \{\pm1\}$ is chosen so that $\det \mathscr{M}^{(1)}>0$.
Let $\mathcal{A}^{(1)}$ be the new matrix part and let $\lambda_j^{(1)}:= (\det \mathscr{M}^{(1)})\lambda_j^{(0)}$ be the eigenvalues of the new matrix.
Then, one sees that
\[
	W(\lambda_1^{(1)};\mathscr{A}^{(1)}) = \Span\{(1,0,-1)\}.
\]

Let us next simplify the characteristic space associated with eigenvalue $\lambda_2^{(1)}$.
Pick
\[
	\vec{q} = {}^t\!(q_1,q_2,q_3) \in W(\lambda_2^{(1)};\mathscr{A}^{(1)})
	\setminus\{0\}.
\]
Note that $q_1q_3-q_2^2<0$.
Since $W(\lambda_2^{(1)};\mathscr{A}^{(1)}) \cap \mathcal{P}_+ = \emptyset$
and $(W(\lambda_1^{(1)};\mathscr{A}^{(1)})\oplus W(\lambda_2^{(1)};\mathscr{A}^{(1)})) \cap \mathcal{P}_+ = \emptyset$, one has
$\max_{t\in\R} ((q_3+t)(q_1-t)-q_2^2)>0$,
which reads as $2|q_2|<|q_1+q_3|$.
In particular, $q_1+q_3\neq0$.
We define
\[
	\begin{bmatrix}
	u_1^{(2)}\\
	u_2^{(2)}
	\end{bmatrix}
	=
	\mathscr{M}^{(2)}
	\begin{bmatrix}
	u_1^{(1)}\\
	u_2^{(1)}
	\end{bmatrix}
\]
with
\[
	\mathscr{M}^{(2)}=\begin{bmatrix}
	\cosh \tau & \sinh \tau  \\
	\sinh \tau & \cosh \tau
	\end{bmatrix} \in SL_2(\R), \quad 
	\tau := \frac12 \tanh^{-1} \( \frac{2q_2}{q_1+q_3} \)\in \R.
\]
Let $\mathscr{A}^{(2)}$ be the new matrix.
A computation shows that
$\mathscr{D}(\mathscr{M}^{(2)}){}^t\!(1,0,-1)={}^t\!(1,0,-1)$ and that
the second component of $\mathscr{D}(\mathscr{M}^{(2)})\vec{q}$ vanishes. Hence,
\[
	W(\lambda_2^{(1)};\mathscr{A}^{(2)})=\Span\{{}^t\!(1,0,-1)\}, \quad
	W(\lambda_2^{(1)};\mathscr{A}^{(2)})=\Span \{ {}^t\!(1,0, -\tilde{q}_3) \}
\]
with some $\tilde{q}_3\in (0,1)\cup (1,\infty)$.
We may suppose $\tilde{q}_3\in (0,1)$ by applying $(	u_1^{(2)},	u_2^{(2)})
\mapsto (	u_2^{(2)},	-u_1^{(2)})$ if necessary,
as in Subcase 1-(a) of the proof of Theorem \ref{thm:standard1}.
We define $\eta_1=\tilde{q}_3/(1-\tilde{q}_3)\in(0,\infty)$ so that
$W(\lambda_2^{(1)};\mathscr{A}^{(2)})=\Span \{ {}^t\!(1+\eta_1,0, -\eta_1) \}$.

We finally adjust the magnitude of the matrix by $u_j^{(3)} := |\lambda_3^{(1)}|^{1/2}u_j^{(2)}$ so that the minimum eigenvalue becomes $-1$.
Let $\mathscr{A}^{(3)}$ be the new matrix. Let $\lambda_j:=\lambda_j^{(1)}/|\lambda_3^{(1)}|$ ($j=1,2$) be the positive eigenvalues of $\mathscr{A}^{(3)}$. We have
\begin{align*}
	W(\lambda_1;\mathscr{A}^{(3)})
	&= \Span\{{}^t\!(1,0,-1)\}, \\
	W(\lambda_2;\mathscr{A}^{(3)})
	&= \Span\{{}^t\!(1+\eta_1,0,-\eta_1)\}, \\
	W(-1;\mathscr{A}^{(3)})&=\Span\{{}^t\!(\eta_2,1,\eta_3)\}
\end{align*}
with $\eta_1>0$ and $\eta_2\eta_3>1$.
Thus, we obtain
\begin{align*}
	\mathscr{A}^{(3)}
	={}& \begin{bmatrix}
	1 & 1+\eta_1 & \eta_2 \\
	0 & 0& 1\\
	-1 & -\eta_1& \eta_3
	\end{bmatrix}
	\begin{bmatrix}
	\lambda_1 & 0 & 0 \\
	0 & \lambda_2 & 0\\
	0 & 0 & -1
	\end{bmatrix}
	\begin{bmatrix}
	1 & 1+\eta_1 & \eta_2 \\
	0 & 0& 1\\
	-1 & -\eta_1& \eta_3
	\end{bmatrix}^{-1} \\
	={}& \begin{bmatrix}
	\lambda_1-(1+\eta_1)(\lambda_1-\lambda_2) & 
	(1+\eta_1)(\eta_2+\eta_3)(\lambda_1-\lambda_2) -(\lambda_1+1)\eta_2
	& -(1+\eta_1)(\lambda_1-\lambda_2) \\
	0 & -1& 0\\
	\eta_1(\lambda_1-\lambda_2) & -\eta_1(\eta_2+\eta_3)(\lambda_1-\lambda_2) -(\lambda_1+1)\eta_3  & \eta_1(\lambda_1-\lambda_2)
	+\lambda_1
	\end{bmatrix}
\end{align*}
with $\eta_1>1$ and $\eta_2\eta_3>1$.
This corresponds to the case $\eta_1>0$ of $\mathscr{A}_{2,2}$.
\end{proof}

\subsection*{Acknowledgements} 
S.M. was  supported by JSPS KAKENHI Grant Numbers 
JP23K20803, JP23K20805, and JP24K00529.
J.S. was supported by JSPS KAKENHI Grant Numbers 
JP21K18588 and JP23K20805.  
K.U. was  supported by JSPS KAKENHI Grant Numbers 
JP 24K00529 and JP24K06805.

\end{document}